\documentclass[11pt,reqno]{article}
\usepackage{amsmath,amsthm,amsfonts,amssymb,amscd}
\usepackage[latin1]{inputenc}                                                       
\usepackage[all]{xy}                                                                
\usepackage[dvips]{graphicx}                                                        
\usepackage{amsmath}                                                                
\usepackage{amsthm}                                                                 
\usepackage{amsfonts}                                                               
\usepackage{amssymb}                                                                

\newtheorem{Theorem}{Theorem}
\newtheorem{Definition}{Definition}
\newtheorem{Afirmacao}{Assertion}

\newtheorem{Corollary}{Corollary}

\newtheorem{Claim}{Claim}
\newtheorem{Question}{Question}
\newtheorem{Problem}{Problem}
\def\bc{\operatorname{{\mathbb C}}}
\def\C{\operatorname{{\mathbb C}}}

\def\GL{\operatorname{{GL}}}
\def\PGL{\operatorname{{PGL}}}
\def\Fix{\operatorname{{Fix}}}
\def\fa{\operatorname{{\mathcal F}}}
\def\Aut{\operatorname{{Aut}}}

\def\F{\operatorname{{\mathcal F}}}
\def\deg{\operatorname{{\rm deg}}}
\def\max{\operatorname{{\rm max}}}
\def\ov\bc{\operatorname{\overline{\mathbb{C}}}}

\title{Vector fields and foliations associated to groups of projective automorphisms}
\author{F. Santos and B. Sc\'ardua}
\date{}

\begin{document}

\maketitle

\begin{abstract}
We introduce and give normal forms for (one-dimensional) Riccati
foliations (vector fields) on $\ov \bc \times \bc P(2)$ and $\ov
\bc \times \ov \bc^n$. These are foliations are characterized by
transversality with the generic fiber of the first projection and
we prove they are conjugate {\em in some invariant Zariski open
subset} to the suspension of a group of automorphisms of the
fiber, $\bc P(2)$ or $\ov \bc^n$, this group called {\it global
holonomy}. Our main result states that given a finitely generated
subgroup $G$ of $\Aut(\bc P (2))$, there is a Riccati foliation on
$\ov \bc \times \bc P(2)$ for which the global holonomy is
conjugate to $G$.
\end{abstract}

\tableofcontents

\section{Introduction}
\label{Section:intro}

Foliations transverse to fibrations are among the  simplest
(constructive)  examples of foliated manifolds, once regarded as
suspensions of group of diffeomorphisms (\cite{camacho},
\cite{Godbillon}). Thus one  expects to perform a nice study of
them in the global theoretic aspect. In the complex algebraic
setting, foliations usually exhibit singularities so this
possibility cannot be excluded. Very representative examples of
either situations are given by the class of {\it Riccati
foliations} (\cite{cidinho}) in dimension two. A very interesting
study is performed in \cite{Pan-Sebastiani} and a complete
reference on the two dimensional case is \cite{Pan-Sebastiani2}.
In this paper we study one-dimensional holomorphic foliations with
singularities which are transverse to a given holomorphic
fibration off some exceptional set in a sense that we shall make
precise below. Let us first recall the classical notion.  Let
$\eta=(E, \pi, B, F)$ be a fibre bundle with total space $E$,
fiber $F$, basis $B$ and projection $\pi\colon E \to B$. A
foliation $\fa$ on $E$ is {\it transverse to $\eta$} if: (1) for
each $p \in E,$ the leaf $L_p$ of $\fa$ with $p\in L_p$ is
transverse to the fiber $\pi^{-1}(q),$ $q=\pi(p);$ (2) $\dim(\fa)
+ \dim(F) = \dim(E);$  and (3) for each leaf $L$ of $\fa,$ the
restriction $\pi|_L: L \to B$ is a covering map. According to
Ehresman (\cite{Godbillon})  if the fiber $F$ is compact, then the
conditions (1) and (2) already  imply (3). In the complex setting
all the objects above are holomorphic by hypothesis and several
are the interesting frameworks (see \cite{Scardua5} for the
complex hyperbolic case). Under the presence of singularities, a
weaker notion must be introduced. We shall say that $\fa$ is {\it
transverse to almost every fiber} of the fibre bundle $\eta$ if
there is an analytic subset $\Lambda(\fa) \subset E$ which is
union of fibers of $\eta$, such that the restriction $\fa_0$ of
$\F$ to $E_0=E\smallsetminus \Lambda$ is transverse to the natural
subbundle $\eta_0$ of $\eta$ having $E_0$ as total space. If
$\Lambda(\fa)$ is minimal with this property then $\Lambda(\fa)$
is called the {\it exceptional set} of $\fa$. By a {\it Riccati
foliation} we mean a foliation $\fa$ as above, for which the
exceptional set $\Lambda(\fa)$ is $\fa$-invariant. In particular
we shall consider the {\it global holonomy} of $\fa$ as the global
holonomy of the restriction $\fa_0$ on $E_0=E\setminus
\Lambda(\fa)$. In the classical situation of Riccati foliations in
$\ov \bc \times \ov  \bc$, the global holonomy is a finitely
generated group of M\"obius transformations, {\it i.e.}, a
finitely generated subgroup of $PSL(2,\bc)$. Using the well-known
classification of M\"obius maps by the set of fixed points (see
Beardon \cite{Beardon}), Lins Neto is able to prove
(\cite{cidinho}) that given a finitely generated subgroup
$G<PSL(2,\bc)$ there is a Riccati foliation in $\ov\bc\times
\ov\bc$ for which the global holonomy is conjugated to $G$.
Similarly, in this work for the case of Riccati foliations on
$\ov\bc \times \bc P(2)$, the study of the global holonomy relies
on the classification of holomorphic diffeomorphisms of $\bc P(2)$
by the set of fixed points. This is the content of
Theorem~\ref{Theorem:ClassificacaoBiholomorfismoEspacoProjetivo}
which  applies  to the  problem of construction of foliations on
$\ov\bc \times \bc P(2)$ with given global holonomy group:

\begin{Theorem}[Synthesis theorem]\label{Theorem:deRealizacao}
    Let $x_0, x_1, \dots, x_k \in \ov\bc$ be points. Let $f_1, \dots,
f_k \in \Aut(\bc P(2))$ be biholomorphisms. Then there is a
Riccati foliation $\fa$ on $\ov\bc \times \bc P(2)$ such that the
invariant fibers of $\fa$ are $\{x_0\} \times \bc P(2), \dots,
\{x_k\} \times \bc P(2)$ and  the global holonomy of $\fa$ is
conjugate to the subgroup of $\Aut(\bc P(2))$ generated by $f_1,
\dots, f_k.$
\end{Theorem}

As mentioned above, our basic motivation comes from the classical
Riccati foliations in dimension two, {\it i.e.},  Riccati
foliations on $\ov\bc \times \ov\bc$, such foliations being given
in affine coordinates by polynomial vector fields of the form $
X\left( x,y \right) = p\left(x\right)\frac{\partial}{\partial x} +
\Bigl( a\left(x\right)y^2 + b\left(x\right)y + c\left(x\right)
\Bigr)\frac{\partial}{\partial y}$. In this Riccati case the fiber
is a (compact) rational Riemman sphere and the exceptional set
$\Lambda\subset {\ov\bc} \times {\ov\bc}$ is a finite union of
vertical projective lines $\{x\}\times \ov\bc$ and is invariant by
the foliation. In this paper we shall mainly work with singular
holomorphic foliations $\F$ on $\ov\bc \times M$ where $M =
\ov\bc^n$ or $M = \C P(2)$, which are transverse to almost every
fiber of $\eta,$ with projection $ \pi:\ov\bc \times M \to
\ov\bc,$ $(x,y) \stackrel{\pi}{\longmapsto} x$. For the case
$M=\ov\bc^n$ these foliations have a natural normal form like in
the Riccati case  as follows:

\begin{Theorem}\label{Theorem:ClassificacaoRiccatiPotenciasEsferaRiemann}
Let $\mathcal{F}$ be a singular holomorphic foliation on $\ov\bc
\times \ov\bc^n$ given by a polynomial vector field $X$ in affine
coordinates  $(x,y) \in \C \times \C^n.$ Suppose $\mathcal{F}$ is
transverse to almost every fiber of the bundle $\eta$ where $\pi:
\ov\bc \times \ov\bc^n \to \ov\bc,$ $\pi(z_1,z_2)= z_1$ is the
projection of $\eta.$ Then $\fa$ is a Riccati foliation and $
X(x,y)=p(x)\frac{\partial}{\partial x}+(q_{1,2}(x) y_1^2 +
q_{1,1}(x) y_1 + q_{1,0}(x))\frac{\partial}{\partial y_1}
+\cdots+(q_{n,2}(x)y_n^2 + q_{n,1}(x) y_n + q_{n,0}(x))
\frac{\partial}{\partial y_n}.$
\end{Theorem}

For the case $M=\bc P(2)$ the classification does not follow an
already established model. Indeed, owing to Okamoto
\cite{okamoto}, given $a, b \in \bc$ the vector field $ X(x,y,z)=
\frac{\partial}{\partial x}+ (z-y^2)\frac{\partial}{\partial y} -
(a+by+yz)\frac{\partial}{\partial z}$ induces a foliation in
$\ov\bc \times \bc P(2)$ which is transverse to every fiber
$\{x\}\times \bc P(2)$ except for $x=\infty$. Our normal form,
 englobing   this class of examples,  is as follows:

\begin{Theorem}\label{Theorem:ClassificacaoRiccatiEspacoProjetivo}
Let $\mathcal{F}$ be a singular holomorphic foliation on $\ov\bc
\times \bc P(2)$ given by a polynomial vector field $X$ in affine
coordinates $(x,y,z)$ on $\C \times \C^2.$ If $\mathcal{F}$ is
transverse to almost every fiber of the fibre bundle $\eta$ where
$\pi: \ov\bc \times \bc P(2) \to \ov\bc,$ $\pi(z_1,z_2)=z_1$ is
the projection of $\eta,$ then $\fa$ is a Riccati foliation and
$X(x,y,z)= p(x)\frac{\partial}{\partial x}+
Q(x,y,z)\frac{\partial}{\partial y} +
R(x,y,z)\frac{\partial}{\partial z};$ in affine coordinates
$(x,y,z) \in \C \times \C^2 \hookrightarrow \ov\bc \times \bc
P(2)$ where
$$
Q(x,y,z)=A(x)+B(x)y+C(x)z+D(x)yz+E(x)y^2,
$$
$$
R(x,y,z)=a(x)+b(x)y+c(x)z+E(x)yz+D(x)y^2,
$$
and $p, a, b, c, A, B, C, D, E\in \bc[x].$
\end{Theorem}

\noindent{\bf Acknowledgement}: We are  grateful to R. S. Mol,
Julio Canille and L. G. Mendes for reading the original manuscript
and various valuable suggestions.

\section{Classification of automorphisms of $\bc P(2)$}                               

The group of automorphisms of $\C P(n)$ is induced by the general
linear group, that is, $\Aut(\C P(n)) \cong \PGL(n+1,\mathbb{C})$
(\cite{griffiths}),  it identifies an isomorphism $T\colon
\bc^{n+1} \to \bc^{n+1}$ with the biholomorphism of   the complex
projective space $[T]$ defined by: if  $r \subset
\mathbb{C}^{n+1}$ is a complex line contains $0 \in
\mathbb{C}^{n+1}$, then $s = T(r)$ is a complex line contains $0
\in \mathbb{C}^{n+1}$ and we consider $[T]: \mathbb{C}P(n)
\rightarrow \mathbb{C}P(n)$ given by $[T](r \smallsetminus \{0\})
= s \smallsetminus \{0\}.$

Aiming the study of Riccati foliations on $\ov\bc \times \bc P(2)$
through the global holonomy we perform  the classification of
holomorphic diffeomorphisms of $\bc P(2)$ by the set of fixed
points. This the content of the following result:

\begin{Theorem}\label{Theorem:ClassificacaoBiholomorfismoEspacoProjetivo}
If $f: \bc P(2) \rightarrow \bc P(2)$ is a biholomorphism and
$\Sigma(f)$ denotes its set of fixed points then we have the
following six possibilities:

\begin{itemize}

\item[{\rm 1.}] $\Sigma(f)$ has pure dimension zero and is a set
of one, two or three points.

\item[{\rm 2.}] $\Sigma(f)$ has pure dimension one and consists of
two projective lines.

\item[{\rm 3.}] $\Sigma(f)$  consists of one point and two
projective lines.

\item[{\rm 4.}] $\Sigma(f)$ has dimension two and $\Sigma(f)=\bc
P(2)$.

\end{itemize}

In particular, $f$ is conjugate in $\Aut(\bc P(2))$ to a map $g
\in\Aut(\mathbb{C}P(2))$ of the form $g(x:y:z)=(\lambda_0 x+y:
\lambda_0 y+z: \lambda_0 z), \, g(x:y:z)=(\lambda_0 x+y: \lambda_0
y: \lambda_1 z), \, g(x:y:z)=(\lambda_0 x: \lambda_1 y: \lambda_2
z), \, g(x:y:z)=(\lambda_0 x: \lambda_0 y: \lambda_1 z), \,
g(x:y:z)=(\lambda_0 x+y: \lambda_0 y: \lambda_0 z), \,
g(x:y:z)=(x:y:z)$, \,  where $\lambda_0, \lambda_1, \lambda_2 \in
\mathbb{C} \smallsetminus \{ 0 \},$ respectively.
\end{Theorem}

\begin{proof}[Proof of
Theorem~\ref{Theorem:ClassificacaoBiholomorfismoEspacoProjetivo}]                   

If $f \in \Aut(\bc P(2)),$ then there is $A = ( a_{ij} )_{3 \times
3} \in \GL(3,\C)$ such that $f=[A],$ that is, $f(x:y:z)=(a_{11} x
+ a_{12} y + a_{13} z : a_{21} x + a_{22} y + a_{23} z : a_{31} x
+ a_{32} y + a_{33} z).$ We use the Jordan canonical forms and
obtain the classification of automorphisms of $\bc P(2)$ by fixed
points. In fact, there are three possibilities for the
characteristic polynomial of $A,$ $p_A(t),$ in $\mathbb{C}[t]$:
\begin{enumerate}
  \item[(i)] $p_A(t)=(t- \lambda_0)(t- \lambda_1)(t- \lambda_2);$
  \item[(ii)] $p_A(t)=(t- \lambda_0)^2(t- \lambda_1);$
  \item[(iii)] $p_A(t)=(t- \lambda_0)^3,$
\end{enumerate}
where  $\lambda_0, \lambda_1, \lambda_2 \in \mathbb{C}
\smallsetminus \{0\}$ are different.

{\bf Case (i).} The minimal polynomial of $A,$ $m_A(t),$ in
$\C[t]$ is $m_A(t) = p_A(t).$ Then there is $P \in \GL(3,\C)$ such
that $A = P^{-1} J P$ where
$$
J= \left[
\begin{array}{ccc}
\lambda_0 & 0 & 0 \\
0 & \lambda_1 & 0 \\
0 & 0 & \lambda_2 \\
\end{array}
\right]
$$
is the Jordan canonical form of $A.$ Therefore $f$ is conjugate to
$[J]$ because $f=[A]=[P^{-1}][J][P].$ We consider $g \equiv [J],$
that is, $g: \bc P(2) \rightarrow \bc P(2)$ defined by
$g(x:y:z)=(\lambda_0 x: \lambda_1 y: \lambda_2 z)$ with
$\lambda_0, \lambda_1, \lambda_2 \in \mathbb{C} \smallsetminus
\{0\}.$ We shall determinate the fixed points of $g.$ First, we
recall that $\bc P(2)$ is a complex manifold defined by the atlas
$\{(E_j, \varphi_j)\}_{j\in \{0,1,2\}}$ where
$$
E_j=\{(z_0:z_1:z_2)\in \bc P(2);\ z_j \neq 0\},
$$
and $\varphi_j:E_j \rightarrow \C^2\,$ is defined by
$\varphi_0(z_0:z_1:z_2)=\left(\frac{z_1}{z_j},\frac{z_2}{z_j}
\right), \varphi_1:E_1 \rightarrow \C^2,\
\varphi_1(z_0:z_1:z_2)=\left( \frac{z_0}{z_1},\frac{z_2}{z_1}
\right), \varphi_2:E_2 \rightarrow \C^2,\ \
\varphi_2(z_0:z_1:z_2)=\left(\frac{z_0}{z_2},\frac{z_1}{z_2}\right).$
Observe that $f$ is conjugate to $g$ so $f$ has the same numbers
of fixed points that $g.$ Now we obtain the points fixed by $g.$
First, we consider the points $(x:y:1) \in \bc P(2).$ In this case
we have $g(x:y:1)=(\lambda_0 x: \lambda_1 y: \lambda_2)$ and the
application $G: \mathbb{C}^2 \rightarrow \mathbb{C}^2$ defined by
$G(x,y)= \left( \frac{\lambda_0}{\lambda_2}x,
\frac{\lambda_1}{\lambda_2}y \right).$ We obtain the following
commutative diagram:
$$
\xymatrix{E_2 \ar[r]^f \ar[d]_{\varphi_2} & E_2 \ar[d]^{\varphi_2}\\
\C^2 \ar[r]_F & \C^2 \\}
$$

Therefore the fixed points of $g \in\Aut(\bc P(2))$ of the form
$(x:y:1)$ are given by the solutions of the following system
$$
\left\{
\begin{array}{ccc}
\frac{\lambda_0}{\lambda_2}x &=& x \\
\frac{\lambda_1}{\lambda_2}y &=& y \\
\end{array}
\right.
$$
and the point $(0,0)$ is this solution so $(0:0:1)$ is a fixed
point by $g.$ By analogy with it we consider the points of the
following form $(x:1:z) \in \bc P(2).$ Now we have
$g(x:1:z)=(\lambda_0x:\lambda_1:\lambda_2z)$ and $G:\C^2 \to \C^2$
defined by $G(x,z)=\left( \frac{\lambda_0}{\lambda_1}x,
\frac{\lambda_2}{\lambda_1}z \right)$ such that the following
diagram is commutative:
$$
\xymatrix{E_1 \ar[r]^f \ar[d]_{\varphi_1} & E_1 \ar[d]^{\varphi_1}\\
\C^2 \ar[r]_F & \C^2 \\}
$$

Notice that the fixed points of $G$ are given by the solutions of
the system
$$
\left\{
\begin{array}{ccc}
x &=& \frac{\lambda_0}{\lambda_1}x \\
z &=& \frac{\lambda_2}{\lambda_1}z \\
\end{array}
\right.
$$
On the other hand this system have the solution $(0,0)$ only.
Therefore $(0:1:0)$ is another fixed point by $g.$ And we consider
the points of the following form $(1:y:z) \in \bc P(2),$ too. We
obtain $g(1:y:z)=(\lambda_0:\lambda_1y:\lambda_2z)$ and $G:
\mathbb{C}^2 \rightarrow \mathbb{C}^2$ defined by $G(y,z)=\left(
\frac{\lambda_1}{\lambda_0}y, \frac{\lambda_2}{\lambda_0}z
\right)$ such that the following diagram is commutative:
$$
\xymatrix{E_0 \ar[r]^f \ar[d]_{\varphi_0} & E_0 \ar[d]^{\varphi_0}\\
\C^2 \ar[r]_F & \C^2 \\}
$$

And the fixed points of $G$ are given by the solutions of the
following system
$$
\left\{
\begin{array}{ccc}
y &=& \frac{\lambda_1}{\lambda_0}y \\
z &=& \frac{\lambda_2}{\lambda_0}z \\
\end{array}
\right.
$$
Now we have the point $(0,0)$ the only solution. Then $(1:0:0)$ is
a fixed point by $g.$ Therefore the fixed points of $g$ are
$$
\Fix(g)=\{ (1:0:0),(0:1:0),(0:0:1) \}.
$$

By analogy with these ideas in the other cases we obtain:

{\bf Case (ii)}. There are two possibilities for the minimal
polynomial of $A,$ $m_A(t),$ in $\C[t]:$
\begin{enumerate}
\item[(ii.1)] $m_A(t)=(t- \lambda_0)(t- \lambda_1);$
\item[(ii.2)] $m_A(t)=(t- \lambda_0)^2(t- \lambda_1)=p_A(t),$
\end{enumerate}
where $\lambda_0, \lambda_1 \in \mathbb{C} \smallsetminus \{0\},$
$\lambda_0 \neq \lambda_1.$ In both of them there is $P \in
\GL(3,\C)$ such that $A=P^{-1}JP$ where $J$ is the Jordan
canonical form of $A.$ Then $f=[A]=[P^{-1}][J][P],$ that is, $f$
is conjugate to $[J].$ Therefore

{\bf Case (ii.1).} In this case we have
$$
J= \left[
\begin{array}{ccc}
\lambda_0 & 0 & 0 \\
0 & \lambda_0 & 0 \\
0 & 0 & \lambda_1 \\
\end{array}
\right].
$$
and then $f$ is conjugate to $g\equiv [J],$ that is,
$g(x:y:z)=(\lambda_0 x : \lambda_0 y : \lambda_1 z).$ Let us study
the fixed points of $g.$ We consider first the points the
following form: $(x:y:1)\in \bc P(2).$ We obtain
$g(x:y:1)=(\lambda_0 x : \lambda_0 y : \lambda_1)$ and
$G:\mathbb{C}^2 \rightarrow \mathbb{C}^2$ defined by
$G(x,y)=\left( \frac{\lambda_0}{\lambda_1}x,
\frac{\lambda_0}{\lambda_1}y \right)$ such that the following
diagram is commutative:
$$
\xymatrix{E_2 \ar[r]^g \ar[d]_{\varphi_2} & E_2 \ar[d]^{\varphi_2}\\
\C^2 \ar[r]_G & \C^2 \\}
$$

Then the fixed points of $G$ are the solutions the following
system
$$
\left\{
\begin{array}{ccc}
\frac{\lambda_0}{\lambda_1}x &=& x \\
\frac{\lambda_0}{\lambda_1}y &=& y \\
\end{array}
\right.
$$
and note that $(0,0)$ is this solution. Therefore $(0:0:1)$ is
fixed point by $g \in \Aut(\bc P(2)).$ Now we consider the points
the following form: $(x:1:z) \in \bc P(2).$ We obtain
$g(x:1:z)=(\lambda_0 x : \lambda_0 : \lambda_1 z)$ and
$G:\mathbb{C}^2 \rightarrow \mathbb{C}^2$ defined by
$G(x,z)=\left( x, \frac{\lambda_1}{\lambda_0}z \right)$ such that
the following diagram is commutative:
$$
\xymatrix{E_1 \ar[r]^g \ar[d]_{\varphi_1} & E_1 \ar[d]^{\varphi_1}\\
\C^2 \ar[r]_G & \C^2 \\}
$$

The fixed points of $G$ are given by the solutions the following
system
$$
\left\{
\begin{array}{ccl}
x &=& x \\
z &=& \frac{\lambda_1}{\lambda_0}z \\
\end{array}
\right.
$$
and we have the following solutions: $\{(x,0)\in \C^2;\ x \in \C
\}.$ Therefore $\Fix_2(g)=\{ (x:1:0) \in \bc P(2);\ x \in
\mathbb{C} \}$ are fixed points of $g.$ At the end we consider the
points of the form $(1:y:z) \in \bc P(2).$ Then we have $g(1:y:z)$
$=(\lambda_0: \lambda_0 y : \lambda_1 z)$ and $G: \mathbb{C}^2
\rightarrow \mathbb{C}^2$ defined by $ g(y,z)=\left( y,
\frac{\lambda_1}{\lambda_0}z \right)$ such that
$$
\xymatrix{E_0 \ar[r]^g \ar[d]_{\varphi_0} & E_0 \ar[d]^{\varphi_0}\\
\C^2 \ar[r]_G & \C^2 \\}
$$
commute. The fixed points of $G$ are the solutions of the
following system
$$
\left\{
\begin{array}{ccc}
y &=& y \\
z &=& \frac{\lambda_1}{\lambda_0}z \\
\end{array}
\right.,
$$
that is, the points $\{(y,0) \in \mathbb{C}^2;\ y \in
\mathbb{C}\}.$ Therefore the points $\Fix_3(g) = \{(1:y:0)\in \bc
P(2) ;\ y \in \mathbb{C}\}.$ are fixed by $g,$ too. Then in this
case the fixed points of $g$ are two projective lines $\Fix_2(g)$
and $\Fix_3(g)$ and one point $(0:0:1) \in \bc P(2).$

{\bf Case (ii.2).} In this case we obtain
$$
J= \left[
\begin{array}{ccc}
\lambda_0 & 1 & 0 \\
0 & \lambda_0 & 0 \\
0 & 0 & \lambda_1 \\
\end{array}
\right]
$$
and $f$ is conjugate by $g = [J],$ that is, $g(x:y:z)=(\lambda_0 x
+ y : \lambda_0 y : \lambda_1 z).$ Let us study of the fixed
points of $g.$ First, we consider the points of the form
$(x:y:1)\in \bc P(2).$ Then we have $g(x:y:1)=(\lambda_0 x + y :
\lambda_0 y : \lambda_1 )$ and $G:\mathbb{C}^2 \rightarrow
\mathbb{C}^2$ defined by $G(x,y)=\left(
\frac{\lambda_0}{\lambda_1} x + \frac{1}{\lambda_1}y ,
\frac{\lambda_0}{\lambda_1}y \right)$ such that the following
diagram is commutative
$$
\xymatrix{E_2 \ar[r]^g \ar[d]_{\varphi_2} & E_2 \ar[d]^{\varphi_2}\\
\C^2 \ar[r]_G & \C^2 \\}
$$

The fixed points of $G$ are the solutions of the following system
$$
\left\{
\begin{array}{ccccc}
\frac{\lambda_0}{\lambda_1}x & + & \frac{1}{\lambda_1}y & = & x \\
& & \frac{\lambda_0}{\lambda_1}y & = & y \\
\end{array}
\right.
$$
and this is the point $(0,0) \in \C^2$ only. Then $(0:0:1) \in \bc
P(2)$ is a fixed point by $g\in\Aut(\bc P(2)).$ On the other hand
we consider the points of the following form $(x:1:z)\in \bc
P(2).$ We obtain $g(x:1:z)=(\lambda_0 x + 1 : \lambda_0 :
\lambda_1 z)$ and $G:\mathbb{C}^2 \rightarrow \mathbb{C}^2$
defined by $G(x,z)=\left( x + \frac{1}{\lambda_0} ,
\frac{\lambda_1}{\lambda_0} z\right)$ such that
$$
\xymatrix{E_1 \ar[r]^g \ar[d]_{\varphi_1} & E_1 \ar[d]^{\varphi_1}\\
\C^2 \ar[r]_G & \C^2 \\}
$$
commute. Notice that the fixed points of  $G$ are the solutions of
the following system
$$
\left\{
\begin{array}{ccccc}
x & = & x & + & \frac{1}{\lambda_0} \\
z & = & \frac{\lambda_1}{\lambda_0}z & & \\
\end{array}
\right.
$$

Observe that there are not solutions of this system, then there
are not fixed points of $g \in\Aut(\bc P(2))$ of the form
$(x:1:z)\in \bc P(2).$ Now we consider points of the form $(1:y:z)
\in \bc P(2).$ We obtain $g(1:y:z)=(\lambda_0 + y : \lambda_0 y :
\lambda_1 z)$ and if $y \neq -\lambda_0,$ then we have
$G:\mathbb{C}^2 \rightarrow \mathbb{C}^2$ defined by
$G(y,z)=\left(\frac{\lambda_0 y}{\lambda_0 + y},\frac{\lambda_1
z}{\lambda_0 + y} \right)$ such that the following diagram is
commutative
$$
\xymatrix{E_0 \ar[r]^g \ar[d]_{\varphi_0} & E_0 \ar[d]^{\varphi_0}\\
\C^2 \ar[r]_G & \C^2 \\}
$$

Now note that if $y = -\lambda_0$ then there are not fixed points
of $g.$ The fixed points of $G$ are given by the solutions of the
system

$$
\left\{
\begin{array}{ccc}
y & = & \frac{\lambda_0 y}{\lambda_0 + y} \\
z & = & \frac{\lambda_1 z}{\lambda_0 + y} \\
\end{array}
\right.
$$
that is, $(0,0)$ is the fixed point by $G.$ Therefore $(1:0:0)$ is
a fixed points of $g$ are $(1:0:0)$ e $(0:0:1).$

{\bf Case (iii).} There are three possibilities for the minimal
polynomial of $A$ in $\C[t]$:
\begin{enumerate}
\item[(iii.1)] $m_A(t)=t-\lambda_0;$
\item[(iii.2)] $m_A(t)=(t-\lambda_0)^2;$
\item[(iii.3)] $m_A(t)=(t-\lambda_0)^3=p_A(t),$
\end{enumerate}
where $\lambda_0 \in \C \smallsetminus \{0\}.$ In all
possibilities there is $P \in \GL(3,\C)$ such that $A=P^{-1} J P$
where $J$ is the Jordan canonical form of $A.$ Then
$f=[J]=[P^{-1}][J][P]$ is conjugate to $[J].$

{\bf Case (iii.1).} In this case we obtain
$$
J= \left[
\begin{array}{ccc}
\lambda_0 & 0 & 0 \\
0 & \lambda_0 & 0 \\
0 & 0 & \lambda_0 \\
\end{array}
\right]
$$
and then $g=[J],$ that is, $g(x:y:z)=(\lambda_0 x : \lambda_0 y :
\lambda_0 z).$ Therefore $g$ is the identity application and all
$\bc P(2)$ are fixed by $g.$

{\bf Case (iii.2)} In this case we have
$$
J= \left[
\begin{array}{ccc}
\lambda_0 & 1 & 0 \\
0 & \lambda_0 & 0 \\
0 & 0 & \lambda_0 \\
\end{array}
\right]
$$
and then $f$ is conjugate to $g(x:y:z)=(\lambda_0 x + y :
\lambda_0 y : \lambda_0 z).$ Let us study the fixed points of $g.$
We consider first the points the following form: $(x:y:1) \in \bc
P(2).$ We obtain $g(x:y:1)=(\lambda_0 x + y : \lambda_0 y :
\lambda_0 )$ and $G:\mathbb{C}^2 \rightarrow \mathbb{C}^2$ defined
by $G(x,y)=\left( x + \frac{1}{\lambda_0}y, y \right)$ such that
$$
\xymatrix{E_2 \ar[r]^g \ar[d]_{\varphi_2} & E_2 \ar[d]^{\varphi_2}\\
\C^2 \ar[r]_G & \C^2 \\}
$$
commute. The fixed points of $G$ are the solutions of the
following system
$$
\left\{
\begin{array}{ccccc}
x & = & x & + & \frac{1}{\lambda_0}y \\
y & = & y & &\\
\end{array}
\right.
$$
and this are $\{ (x,0) \in \mathbb{C}^2;\ x \in \mathbb{C}.$ Then
the points $ \Fix_1(g)=\{ (x:0:1)\in \bc P(2);\ x\in \mathbb{C}
\}$ are fixed by $g.$ Now we consider the points of the form
$(x:1:z)\in \bc P(2).$ We have $g(x:1:z)=(\lambda_0 x + 1 :
\lambda_0 : \lambda_0 z)$ and $G : \mathbb{C}^2 \rightarrow
\mathbb{C}^2$ defined by $G(x,z)=\left( x + \frac{1}{\lambda_0}, z
\right)$ such that
$$
\xymatrix{E_1 \ar[r]^g \ar[d]_{\varphi_1} & E_1 \ar[d]^{\varphi_1}\\
\C^2 \ar[r]_G & \C^2 \\}
$$
commute. Then fixed points of $G$ are the solution of the system
$$
\left\{
\begin{array}{ccccc}
x & = & x & + & \frac{1}{\lambda_0} \\
z & = & z & & \\
\end{array}
\right.
$$

Notice that this system doesn't have solutions. Therefore there
are not fixed points of $g$ of the form $(x:1:z)\in \bc P(2).$ And
now we consider the points of the form $(1:y:z)\in\bc P(2).$ We
have $g(1:y:z)=(\lambda_0 + y : \lambda_0 y : \lambda_0 z)$ and
$G:\mathbb{C}^2\rightarrow\mathbb{C}^2$ defined by $G(y,z)=\left(
\frac{\lambda_0 y}{\lambda_0 + y},\frac{\lambda_0 z}{\lambda_0 +
y} \right)$ such that the following diagram is commutative
$$
\xymatrix{E_0 \ar[r]^g \ar[d]_{\varphi_0} & E_0 \ar[d]^{\varphi_0}\\
\C^2 \ar[r]_G & \C^2 \\}
$$
if $y \neq -\lambda_0.$ Notice that if $y=-\lambda_0$ then there
are not fixed points of $g.$ The fixed points of $G$ are given by
the solutions of the system
$$
\left\{
\begin{array}{ccc}
y & = & \frac{\lambda_0 y}{\lambda_0 + y} \\
z & = & \frac{\lambda_0 z}{\lambda_0 + y} \\
\end{array}
\right.
$$
and these are the points $\{(0,z)\in \mathbb{C}^2;\ z\in
\mathbb{C}\}.$ Therefore $\Fix_3(g)=\{(1:0:z)\in \bc P(2);\
z\in\mathbb{C} \}$ is a subset of the fixed points of $g\in
Aut(\bc P(2)).$ Therefore the fixed points of $g$ in this case are
two projective lines
$$
\Fix(g)=\{ (x:0:1)\in \bc P(2);\ x\in \mathbb{C} \} \cup \{
(1:0:z)\in \bc P(2);\ z\in\mathbb{C} \}.
$$

{\bf Case (iii.3).} We have
$$
J= \left[
\begin{array}{ccc}
\lambda_0 & 1 & 0 \\
0 & \lambda_0 & 1 \\
0 & 0 & \lambda_0 \\
\end{array}
\right],
$$
and $f$ is conjugate to $g(x:y:z)=(\lambda_0 x + y : \lambda_0 y +
z : \lambda_0 z).$ Let us study the fixed points of $g.$ First, we
consider the points of the following form $(x:y:1) \in \bc P(2).$
Then we have $g(x:y:1)=(\lambda_0 x + y : \lambda_0 y + 1 :
\lambda_0 )$ and $G:\mathbb{C}^2 \rightarrow \mathbb{C}^2$ defined
by $G(x,y)=\left( x+ \frac{y}{\lambda_0},y+ \frac{1}{\lambda_0}
\right)$ such that
$$
\xymatrix{E_2 \ar[r]^g \ar[d]_{\varphi_2} & E_2 \ar[d]^{\varphi_2}\\
\C^2 \ar[r]_G & \C^2 \\}
$$
commute. The fixed points of $G$ are given by the system
$$
\left\{
\begin{array}{ccccc}
x & = & x & + & \frac{1}{\lambda_0}y \\
y & = & y & + & \frac{1}{\lambda_0} \\
\end{array}
\right.
$$

Therefore there are not fixed points of $g \in\Aut(\bc P(2))$ of
the form $(x:y:1) \in \bc P(2)$ in this case, because there are
not solutions for this system. Now we consider points of the form
$(x:1:z) \in \bc P(2).$ We obtain $g(x:1:z)=(\lambda_0 x + 1 :
\lambda_0  + z : \lambda_0 z)$ and $G:\mathbb{C}^2 \rightarrow
\mathbb{C}^2$ defined by $G(x,z)=\left(\frac{\lambda_0 x +
1}{\lambda_0 + z}, \frac{\lambda_0 z}{\lambda_0 + z} \right)$ such
that
$$
\xymatrix{E_1 \ar[r]^g \ar[d]_{\varphi_1} & E_1 \ar[d]^{\varphi_1}\\
\C^2 \ar[r]_G & \C^2 \\}
$$
commute if $z \neq - \lambda_0.$ Notice that if $z=- \lambda_0$
then there are not fixed points of $g\in \Aut (\bc P(2)).$ Then
the fixed points of $g$ are given by the fixed points of $G$ and
these are given by the solutions of the system
$$
\left\{
\begin{array}{ccc}
x & = & \frac{\lambda_0 x + 1}{\lambda_0 z} \\
z & = & \frac{\lambda_0 z}{\lambda_0 + z} \\
\end{array}
\right.
$$

Therefore there are not fixed points of $g$ in this case. By
analogy with it we consider the points of the form $(1:y:z) \in
\bc P(2).$ We obtain $g(1:y:z)=(\lambda_0 + y : \lambda_0 y + z :
\lambda_0 z)$ and $G : \C^2 \to \C^2$ defined by
$G(y,z)=\left(\frac{\lambda_0 y + z}{\lambda_0 + y},
\frac{\lambda_0 z}{\lambda_0 + y} \right)$ such that
$$
\xymatrix{E_0 \ar[r]^g \ar[d]_{\varphi_0} & E_0 \ar[d]^{\varphi_0}\\
\C^2 \ar[r]_G & \C^2 \\}
$$
commute if $y \neq - \lambda_0.$ Observe that if $y= -\lambda_0$
then there are not fixed points of $g.$ The fixed points of $G$
are given by the system
$$
\left\{
\begin{array}{ccc}
y & = & \frac{\lambda_0 y + z}{\lambda_0 + y} \\
z & = & \frac{\lambda_0 z}{\lambda_0 + y} \\
\end{array}
\right.
$$

And then $(0,0)$ is the fixed point by $G.$ Therefore $(1:0:0)$ is
the only fixed point by $g.$ And we have finished the proof of
Theorem~\ref{Theorem:ClassificacaoBiholomorfismoEspacoProjetivo}
\end{proof}

Now we obtain the following

\begin{Definition}{\rm
We shall say that the biholomorphism $f: \bc P(2) \to \bc P(2)$
except for identity application is:
\begin{enumerate}
\item[(i)] of type P1 if one point is fixed by $f;$ \item[(ii)] of
type P2 if two points are fixed by $f;$ \item[(iii)] of type P3 if
three points are fixed by $f;$ \item[(iv)] of type R2 if two
projective lines are fixed by $f;$ \item[(v)] of type P1R2 if one
point and two projective lines are fixed by $f.$
\end{enumerate}
}
\end{Definition}

\section{Construction of Riccati foliations on $\ov\bc \times \bc P(2)$}               

In this section we address the following question:

\begin{Question} Let be given elements $f_1, \dots, f_k$ of the group
$\Aut(\bc P(2))$. Is there a Riccati foliation $\fa$ on $\ov\bc
\times \bc P(2)$ such that the global holonomy of $\fa$ is
conjugate to the subgroup of $\Aut(\bc P(2))$ generated by $f_1,
\dots, f_k$? \end{Question}

We proceed similarly to   \cite{cidinho},  that is, we  construct
a singular holomorphic foliation $\fa$ on $\ov\bc \times \bc P(2)$
by gluing together local foliations transverse to almost every
fiber and given in a neighborhood of the invariant fibers by
suitable local models given in terms of the normal form of the
corresponding holonomy map as in
Theorem~\ref{Theorem:ClassificacaoBiholomorfismoEspacoProjetivo}.

\begin{proof} [Proof of Theorem~\ref{Theorem:deRealizacao}]                            
Let $f_0=(f_1\circ \cdots \circ f_k)^{-1}$ be a biholomorphism and
let $x_0 =0, x_1, \dots, x_k$ be points in $\C.$ For each $j \in
\{0,1,\dots,k\}$ let $D_j$ be a disk of radius $r > 0$ and center
$x_j$ such that $|x_i - x_j| > 2r$ for all $i \neq j,\ 0 \leq i,j
\leq k.$ For each $j\in \{1,\dots,k\}$ we choose $x'_j=x_j +
\frac{r}{2} \in D_j \smallsetminus \{x_j\}$ and $x''_j=\frac{r}{2}
\exp(\frac{2 \pi \sqrt{-1}(j-1)}{k}) \in D_0 \smallsetminus
\{0\}.$

Let $\alpha_1, \dots, \alpha_k :[0,1] \to \C$ be simple curves
such that (i) $\alpha_j(0)=x''_j$ and $\alpha_j(1)=x'_j;$ (ii)
$\alpha_j([0,1]) \cap D_i =\emptyset$ if $i \neq j,\ i \neq 0;$
(iii) $\alpha_i([0,1]) \cap \alpha_j([0,1]) =\emptyset$ if $i \neq
j;$ (iv) for every $j \in \{1,\dots, k\},$ $\alpha_j([0,1]) \cap
D_0$ and $\alpha_j([0,1]) \cap D_j$ are segments of straight lines
contained in diameters of $D_0$ and $D_j$ respectively.

Let $A_1, \dots, A_k$ be tubular neighborhoods of $\alpha_1,
\dots, \alpha_k$ respectively such that (v) $A_j \cap
D_i=\emptyset$ if $i \neq j,\ i \neq 0;$ (vi) $A_i \cap
A_j=\emptyset$ if $i \neq j;$ (vii) $A_j \cap D_0$ and $A_j \cap
D_j$ are contained in sectors of $D_0$ and $D_j,$ $1 \leq j \leq
k,$ respectively.

Let $U=(\bigcup^k_{j=1}A_j) \cup (\bigcup^k_{j=0}D_j)$ be a set
and let $\gamma = \partial U$ be a simple curve. Let $T$ be a
tubular neighborhood of $\gamma$ and let $V=(\ov\bc \smallsetminus
U) \cup T$ be a set. Then $\{ A_1, \dots, A_k, D_0, \dots, D_k, V
\}$ is a covering of $\ov\bc$ by open sets. For every $j \in \{1,
\dots, k\}$ we consider affine coordinates $(x,U_j,V_j)$ in $A_j
\times \C^2 \hookrightarrow A_j \times \bc P(2),\ x \in A_j,\
(U_j,V_j) \in \C^2.$ For each $i \in \{0,1,\dots,k \}$ we consider
affine coordinates $(x,u_i,v_i)$ in $D_i \times \C^2
\hookrightarrow D_i \times \bc P(2),\ x \in D_i,\ (u_i,v_i) \in
\C^2.$ Put affine coordinates $(w,y_1,y_2)$ in $V \times \C^2$
where $w= \frac{1}{x} \in V$ and $(y_1,y_2)\in \C^2.$

We take in each set of the form $A_j \times \C^2,$ $V \times \C^2$
and $D_i \times \C^2$ a local model of foliation and glue them
together. The local models are as follows:
\begin{enumerate}

\item In $A_j \times \bc P(2)$ we consider the horizontal
foliation whose leaves are of the form $A_j \times \{p\}, \ p \in
\bc P(2)$ for each $j \in \{1, \dots, k \}.$

\item In $V \times \bc P(2)$ we consider the horizontal foliation
whose leaves are of the form $V \times \{p\},\ p \in \bc P(2).$

\item In $D_i \times \C^2$ we consider the singular holomorphic
foliation $\fa_i$ induced by the vector field $X_i$ in $D_i \times
\C^2$ for every $i \in \{ 0,1,\dots,k \}.$ Put $l \in \{0, 1,
\dots, k \}.$ There exists an affine coordinate such that $f_l:
E_0 \to E_0$ can be written in one of the following forms:
\begin{enumerate}
\item[(a)] $f_l(u,v)=(u + \mu_l v, v + \mu_l)$ if $f_l$ is of type
P1. \item[(b)] $f_l(u,v)=(\mu_l u + \nu_l v, \mu_l v)$ if $f_l$ is
of type P2. \item[(c)] $f_l(u,v)=(\lambda'_l u, \lambda''_l v)$ if
$f_l$ is P3. \item[(d)] $f_l(u,v)=(\lambda''_l u, \lambda''_l v)$
if $f_l$ is R2. \item[(e)] $f_l(u,v)=(u+ \nu_l v, v)$ if $f_l$ is
of type P1R2.
\end{enumerate}
where $\lambda'_l, \lambda''_l, \mu_l, \nu_l \in \C \smallsetminus
\{0\}$ are different.
\end{enumerate}

$\bullet$ In  case (c) ((d) respectively) we consider the singular
holomorphic foliation $\fa_j$ on $D_j \times \C^2$ given by the
vector field
\begin{equation}\label{Campo:EstruturaLocal1}
X_j(x,u_j,v_j)=(x-x_j)\frac{\partial}{\partial x} + \alpha'_j u_j
\frac{\partial}{\partial u_j}+ \alpha''_j v_j
\frac{\partial}{\partial v_j}
\end{equation}
where $\exp(2 \pi \sqrt{-1}\alpha'_j)=\lambda'_j$ and $\exp(2 \pi
\sqrt{-1}\alpha''_j)=\lambda''_j.$ (In case (d) the foliation
$\fa_j$ is given by
$X_j(x,u_j,v_j)=(x-x_j)\frac{\partial}{\partial x} + \alpha''_j
u_j \frac{\partial}{\partial u_j}+ \alpha''_j v_j
\frac{\partial}{\partial v_j}$ where $\alpha''_j$ and
$\lambda''_j$ are the same, respectively.)

Let $\gamma_j(\theta)=(r_j \exp(\sqrt{-1} \theta) + x_j,0,0),$ $0
\leq \theta \leq 2 \pi$ be a curve where $0 < r_j < r.$ Let
$\Sigma_j= \{p_j\} \times \C^2,$ $p_j \in \gamma_j([0,2 \pi]).$

\begin{Afirmacao}\label{Afirmacao:EstruturaLocal}
The holonomy transformation of $\fa_j$ associated to $\Sigma_j$
and $\gamma_j$ is of the form $(u_j, v_j) \mapsto (\lambda'_j u_j,
\lambda''_j v_j)$ where the foliation $\fa_j$ on $D_j \times \C^2$
is induced by equation~\ref{Campo:EstruturaLocal1}.
\end{Afirmacao}

In fact, let $\Sigma_j=\{ x_j + r_j\} \times \C^2$ be a local
transverse section and let $p_j=(x_j + r_j ,0 ,0) \in \Sigma_j.$
Suppose $p_1: D_j \times \C^2 \to D_j,$ $p_1(x,y,z)=x.$ Observe
that the fibers $p_1^{-1}(x),$ $x \neq x_j,$ are transverse to
$\fa.$ Let $q=(x_j + r_j, u_j, v_j) \in \Sigma_j$ and let
$\gamma_q(\theta)=(x(\theta), u_j(\theta), v_j(\theta))$ be the
lifting of $\gamma_j$ by $p_1$ with base point $q.$ Therefore
$$
x'(\theta)=p_1(\gamma'_q(\theta))=p_1(\gamma'_j(\theta))=
\sqrt{-1}r_j \exp(\sqrt{-1} \theta),
$$
and, if $Y_j=(u_j, v_j) \in \C^2$ then
$$
\frac{Y'_j}{x'}= \frac{Y'_j}{\sqrt{-1}r_j \exp(\sqrt{-1} \theta)}.
$$
On the other hand, by equation~\ref{Campo:EstruturaLocal1} we have
$$
\frac{dx}{dT}=x-x_j
$$
and
$$
\frac{dY_j}{dT}=\left[
\begin{array}{cc}
\alpha'_j & 0\\
0         & \alpha''_j\\
\end{array}
\right] \cdot \left[
\begin{array}{c}
u_j\\
v_j\\
\end{array}
\right]= A Y_j
$$
so we obtain
$$
\frac{dY_j}{dx}=\frac{dY_j}{dT} \cdot
\frac{dT}{dx}=\frac{\frac{dY_j}{dT}}{\frac{dx}{dT}}=\frac{A Y_j}{x
- x_j}
$$
and we have
$$
\frac{Y'_j}{\sqrt{-1}r_j \exp(\sqrt{-1} \theta)}= \frac{Y'_j}{x'}=
\frac{A Y_j}{r_j \exp(\sqrt{-1} \theta)}
$$
then $Y'_j=\sqrt{-1}A Y_j.$ Notice that the solution of
$Y'_j=\sqrt{-1}AY_j$ such that $Y_j(0)= (u_j, v_j)$ is
$Y_j(\theta)= \exp(\sqrt{-1} \theta A) \cdot Y_j(0).$ Therefore
the holonomy is the biholomorphism $f:\Sigma_j \to \Sigma_j$
defined by
$$
\begin{array}{ccl}
f(u_j, v_j) & = & Y_j(2 \pi)\\
&&\\
            & = & \exp\left( 2 \pi \sqrt{-1} \left[
                                       \begin{array}{cc}
                                       \alpha'_j & 0 \\
                                       0         & \alpha''_j \\
                                       \end{array}
                                       \right] \right)
                                       \cdot
                                       \left[
                                       \begin{array}{c}
                                       u_j\\
                                       v_j \\
                                       \end{array}
                                      \right]\\
&&\\
            & = &\left[
            \begin{array}{cc}
            \exp(2 \pi \sqrt{-1} \alpha'_j) & 0 \\
            0                               & \exp(2 \pi \sqrt{-1} \alpha''_j) \\
            \end{array}
            \right]
            \cdot
            \left[
            \begin{array}{c}
            u_j\\
            v_j \\
            \end{array}
            \right]\\
&&\\
            & = & (\exp(2 \pi \sqrt{-1} \alpha'_j) u_j, \exp(2 \pi \sqrt{-1} \alpha''_j)
            v_j)\\
&&\\
            & = & (\lambda'_j u_j, \lambda''_j v_j)\\
\end{array}
$$
and this proves the assertion. (In case (d) we prove the holonomy
transformation of $\fa_j$ associated to $\Sigma_j$ and $\gamma_j$
is $(u_j, v_j) \mapsto (\lambda''_j u_j, \lambda''_j v_j)$
respectively.)

$\bullet$ In  case (e) we consider the foliation $\fa_j$ on
$D_j\times \C^2$ given by
\begin{equation}\label{Campo:EstruturaLocal3}
X_j(x,u_j,v_j)=(x-x_j)\frac{\partial}{\partial x} + \frac{\nu_j}{2
\pi \sqrt{-1}} v_j \frac{\partial}{\partial u_j}.
\end{equation}
Let $\gamma_j(\theta)=(r_j \exp(\sqrt{-1} \theta) + x_j,0,0),$ $0
\leq \theta \leq 2 \pi$ be a curve where $0 < r_j < r.$ Let
$\Sigma_j= \{p_j\} \times \C^2,$ $p_j \in \gamma_j([0,2 \pi]).$

\begin{Afirmacao}
The holonomy transformation of $\fa_j$ associated to $\Sigma_j$
and $\gamma_j$ is the  following form $(u_j, v_j) \mapsto (u_j +
\nu_j v_j, v_j),$ where the foliation $\fa_j$ on $D_j \times \C^2$
is given by equation~\ref{Campo:EstruturaLocal3}.
\end{Afirmacao}

In fact, let $\Sigma_j=\{ x_j + r_j\} \times \C^2$ be a local
transverse section and let $p_j=(x_j + r_j ,0 ,0) \in \Sigma_j.$
Suppose $p_1: D_j \times \C^2 \to D_j,$ $p_1(x,y,z)=x.$ Notice
that the fibers $p_1^{-1}(x),$ $x \neq x_j$ are transverse to
$\fa.$ Let $q=(x_j + r_j, u_j, v_j) \in \Sigma_j$ and let
$\gamma_q(\theta)=(x(\theta), u_j(\theta), v_j(\theta))$ be the
lifting of $\gamma_j$ by $p_1$ with base point $q.$ Therefore
$x'(\theta)=p_1(\gamma'_q(\theta))=p_1(\gamma'_j(\theta))=
\sqrt{-1}r_j \exp(\sqrt{-1} \theta),$ and, if $Y_j=(u_j, v_j) \in
\C^2$ then
$$
\frac{Y'_j}{x'}= \frac{Y'_j}{\sqrt{-1}r_j \exp(\sqrt{-1} \theta)}.
$$
On the other hand, by equation~\ref{Campo:EstruturaLocal3} we have
$\frac{dx}{dT}=x-x_j$ and
$$
\frac{dY_j}{dT}=\left[
\begin{array}{cc}
0 & \frac{\nu_j}{2 \pi \sqrt{-1}}\\
0 & 0\\
\end{array}
\right] \cdot \left[
\begin{array}{c}
u_j\\
v_j\\
\end{array}
\right]= A Y_j,
$$
so we obtain
$$
\frac{dY_j}{dx}=\frac{dY_j}{dT} \cdot
\frac{dT}{dx}=\frac{\frac{dY_j}{dT}}{\frac{dx}{dT}}=\frac{A Y_j}{x
- x_j}
$$
and we have got
$$
\frac{Y'_j}{\sqrt{-1}r_j \exp(\sqrt{-1} \theta)}= \frac{Y'_j}{x'}=
\frac{A Y_j}{r_j \exp(\sqrt{-1} \theta)}
$$
therefore $Y'_j=\sqrt{-1}A Y_j.$ Observe that the solution of
$Y'_j=\sqrt{-1}AY_j$ with $Y_j(0)= (u_j, v_j)$ is $Y_j(\theta)=
\exp(\sqrt{-1} \theta A) \cdot Y_j(0).$ Therefore the holonomy is
the biholomorphism $f:\Sigma_j \to \Sigma_j$ defined by
$$
\begin{array}{ccl}
f(u_j, v_j) & = & Y_j(2 \pi)\\
&&\\
            & = & \exp\left( 2 \pi \sqrt{-1} \left[
                                       \begin{array}{cc}
                                       0 & \frac{\nu_j}{2 \pi \sqrt{-1}}\\
                                       0 & 0\\
                                       \end{array}
                                       \right] \right)
                                       \cdot
                                       \left[
                                       \begin{array}{c}
                                       u_j\\
                                       v_j \\
                                       \end{array}
                                      \right]\\
&&\\
            & = &\exp \left( \left[
            \begin{array}{cc}
            0 & \nu_j\\
            0 & 0\\
            \end{array}
            \right] \right)
            \cdot
            \left[
            \begin{array}{c}
            u_j\\
            v_j \\
            \end{array}
            \right]\\
&&\\
            & = & \left(
            \left[
            \begin{array}{cc}
            1 & 0 \\
            0 & 1 \\
            \end{array}
            \right]
            +
            \left[
            \begin{array}{cc}
            0 & \nu_j \\
            0 & 0 \\
            \end{array}
            \right]
                  \right)
            \cdot
            \left[
            \begin{array}{c}
            u_j\\
            v_j \\
            \end{array}
            \right]\\
&&\\
            & = &             \left[
            \begin{array}{cc}
            1 & \nu_j \\
            0 & 1 \\
            \end{array}
            \right]
            \cdot
            \left[
            \begin{array}{c}
            u_j\\
            v_j \\
            \end{array}
            \right]\\
&&\\
            & = & (u_j + \nu_j v_j, v_j)\\
\end{array}
$$
and it proves the assertion.

In  case (b) we consider the singular foliation $\fa$ on $D_j
\times \C^2$ given by
\begin{equation}\label{Campo:EstruturaLocal11}
X_j(x,u_j,v_j)=(x-x_j)\frac{\partial}{\partial x} + ( \lambda u_j
+ \frac{\nu}{2 \pi \sqrt{-1} \mu} v_j )\frac{\partial}{\partial
u_j} + \lambda v_j \frac{\partial}{\partial v_j}.
\end{equation}
where $\exp(2 \pi \sqrt{-1} \lambda) = \mu.$

Let $\gamma_j(\theta)=(r_j \exp(\sqrt{-1} \theta) + x_j,0,0),$ $0
\leq \theta \leq 2 \pi$ be a curve where $0 < r_j < r.$ Let
$\Sigma_j= \{p_j\} \times \C^2,$ $p_j \in \gamma_j([0,2 \pi]).$

\begin{Afirmacao}
The holonomy transformation of $\fa_j$ associated to $\Sigma_j$
and $\gamma_j$ is of the form $(u_j, v_j) \mapsto (\mu u_j + \nu
v_j, \mu v_j)$ where the foliation $\fa_j$ on $D_j \times \C^2$ is
induced by equation~\ref{Campo:EstruturaLocal11}.
\end{Afirmacao}

In fact, let $\Sigma_j=\{ x_j + r_j\} \times \C^2$ be a local
transverse section and let $p_j=(x_j + r_j ,0 ,0) \in \Sigma_j.$
Suppose $p_1: D_j \times \C^2 \to D_j,$ $p_1(x,y,z)=x.$ Notice
that the fibers $p_1^{-1}(x),$ $x \neq x_j$ are transverse to
$\fa.$ Let $q=(x_j + r_j, u_j, v_j) \in \Sigma_j$ and let
$\gamma_q(\theta)=(x(\theta), u_j(\theta), v_j(\theta))$ be the
lifting of $\gamma_j$ by $p_1$ with base point $q.$ Therefore
$$
x'(\theta)=p_1(\gamma'_q(\theta))=p_1(\gamma'_j(\theta))=
\sqrt{-1}r_j \exp(\sqrt{-1} \theta),
$$
and, if $Y_j=(u_j, v_j) \in \C^2$ then $ \frac{Y'_j}{x'}=
\frac{Y'_j}{\sqrt{-1}r_j \exp(\sqrt{-1} \theta)}. $ On the other
hand, by equation~\ref{Campo:EstruturaLocal11} we have $
\frac{dx}{dT}=x-x_j $ and
$$
\frac{dY_j}{dT}=\left[
\begin{array}{cc}
\lambda & \frac{\nu}{2 \pi \sqrt{-1} \mu}\\
0 & \lambda \\
\end{array}
\right] \cdot \left[
\begin{array}{c}
u_j\\
v_j\\
\end{array}
\right]= A Y_j,
$$
so we obtain
$$
\frac{dY_j}{dx}=\frac{dY_j}{dT} \cdot
\frac{dT}{dx}=\frac{\frac{dY_j}{dT}}{\frac{dx}{dT}}=\frac{A Y_j}{x
- x_j}
$$
and we have
$$
\frac{Y'_j}{\sqrt{-1}r_j \exp(\sqrt{-1} \theta)}= \frac{Y'_j}{x'}=
\frac{A Y_j}{r_j \exp(\sqrt{-1} \theta)}
$$
therefore $Y'_j=\sqrt{-1}A Y_j.$ Observe that the solution of
$Y'_j=\sqrt{-1}AY_j$ such that $Y_j(0)= (u_j, v_j)$ is
$Y_j(\theta)= \exp(\sqrt{-1} \theta A) \cdot Y_j(0).$ Therefore
the holonomy is the biholomorphism $f:\Sigma_j \to \Sigma_j$
defined by
$$
\begin{array}{ccl}
f(u_j, v_j) & = & Y_j(2 \pi)\\
&&\\
            & = & \exp\left( 2 \pi \sqrt{-1} \left[
                                       \begin{array}{cc}
                                       \lambda & \frac{\nu}{2 \pi \sqrt{-1} \mu}\\
                                       0 & \lambda\\
                                       \end{array}
                                       \right] \right)
                                       \cdot
                                       \left[
                                       \begin{array}{c}
                                       u_j\\
                                       v_j \\
                                       \end{array}
                                      \right]\\
&&\\
            & = & \exp\left(
            \left[
            \begin{array}{cc}
            2 \pi \sqrt{-1} \lambda & 0 \\
            0 & 2 \pi \sqrt{-1} \lambda \\
            \end{array}
            \right]
            +
            \left[
            \begin{array}{cc}
            0 & \frac{\nu}{\mu} \\
            0 & 0 \\
            \end{array}
            \right]
                  \right)
            \cdot
            \left[
            \begin{array}{c}
            u_j\\
            v_j \\
            \end{array}
            \right]\\
&&\\
            & = & \exp\left( \left[
                                       \begin{array}{cc}
                                       2 \pi \sqrt{-1} \lambda & 0 \\
                                       0 & 2 \pi \sqrt{-1} \lambda \\
                                       \end{array}
                                       \right] \right)
                                       \cdot
                                       \exp\left( \left[
                                       \begin{array}{cc}
                                       0 & \frac{\nu}{\mu} \\
                                       0 & 0 \\
                                       \end{array}
                                       \right] \right)
                                       \left[
                                       \begin{array}{c}
                                       u_j\\
                                       v_j \\
                                       \end{array}
                                      \right]\\
&&\\
& = & \left[
                                       \begin{array}{cc}
                                       \mu & 0 \\
                                       0 & \mu \\
                                       \end{array}
                                       \right]
                                       \left(
                                       \left[
                                       \begin{array}{cc}
                                       1 & 0 \\
                                       0 & 1 \\
                                       \end{array}
                                       \right]+
                                        \frac{\nu}{\mu} \left[
                                       \begin{array}{cc}
                                       0 & 1 \\
                                       0 & 0 \\
                                       \end{array}
                                       \right]
                                       \right)
                                       \left[
                                       \begin{array}{c}
                                       u_j\\
                                       v_j \\
                                       \end{array}
                                      \right]\\
&&\\
 & = & \left[
                                       \begin{array}{cc}
                                       \mu & 0 \\
                                       0 & \mu \\
                                       \end{array}
                                       \right]
                                       \left[
                                       \begin{array}{cc}
                                       1 & \frac{\nu}{\mu} \\
                                       0 & 1 \\
                                       \end{array}
                                       \right]
                                       \left[
                                       \begin{array}{c}
                                       u_j\\
                                       v_j \\
                                       \end{array}
                                      \right]\\
&&\\
            & = &             \left[
            \begin{array}{cc}
            \mu & \nu \\
            0 & \mu \\
            \end{array}
            \right]
            \cdot
            \left[
            \begin{array}{c}
            u_j\\
            v_j \\
            \end{array}
            \right]\\
&&\\
            & = & (\mu u_j + \nu v_j, \mu v_j)\\
\end{array}
$$
and this proves the assertion.

In  case (a) we consider the foliation $\fa_j$ on $D_j \times
\C^2$ given by
\begin{equation}\label{Campo:EstruturaLocal12}
X_j(x,u_j,v_j)=(x-x_j)\frac{\partial}{\partial x} + (\frac{\mu}{2
\pi \sqrt{-1}} v_j - \frac{\mu^2}{4 \pi \sqrt{-1}})
\frac{\partial}{\partial u_j} + \frac{\mu}{2 \pi \sqrt{-1}}
\frac{\partial}{\partial v_j}.
\end{equation}
Let $\gamma_j(\theta)=(r_j \exp(\sqrt{-1} \theta) + x_j,0,0),$ $0
\leq \theta \leq 2 \pi$ be a curve where $0 < r_j < r.$ Let
$\Sigma_j= \{p_j\} \times \C^2,$ $p_j \in \gamma_j([0,2 \pi]).$

\begin{Afirmacao}
The holonomy transformation of $\fa_j$ associated to $\Sigma_j$
and $\gamma_j$ is of the form $(u_j, v_j) \mapsto (u_j + \mu v_j,
v_j + \mu).$
\end{Afirmacao}

In fact, let $\Sigma_j=\{ x_j + r_j\} \times \C^2$ be a local
transverse section and let $p_j=(x_j + r_j ,0 ,0) \in \Sigma_j.$
Suppose $p_1: D_j \times \C^2 \to D_j,$ $p_1(x,y,z)=x.$ Notice
that the fibers $p_1^{-1}(x),$ $x \neq x_j$ are transverse to
$\fa.$ Let $q=(x_j + r_j, u_j, v_j) \in \Sigma_j$ and let
$\gamma_q(\theta)=(x(\theta), u_j(\theta), v_j(\theta))$ be the
lifting of $\gamma_j$ by $p_1$ with base point $q.$ Therefore
$$
x'(\theta)=p_1(\gamma'_q(\theta))=p_1(\gamma'_j(\theta))=
\sqrt{-1}r_j \exp(\sqrt{-1} \theta),
$$
and, if $Y_j=(u_j, v_j) \in \C^2$ then
$$
\frac{Y'_j}{x'}= \frac{Y'_j}{\sqrt{-1}r_j \exp(\sqrt{-1} \theta)}.
$$
On the other hand, by equation~\ref{Campo:EstruturaLocal3} we have
$$
\frac{dx}{dT}=x-x_j
$$
and
$$
\frac{dY_j}{dT}=\left[
\begin{array}{cc}
0 & \frac{\mu}{2 \pi \sqrt{-1}}\\
0 & 0\\
\end{array}
\right] \cdot \left[
\begin{array}{c}
u_j\\
v_j\\
\end{array}
\right] + \left[\begin{array}{c}
-\frac{\mu^2}{4 \pi \sqrt{-1}}\\
\frac{\mu}{2 \pi \sqrt{-1}}\\
\end{array} \right]= A Y_j + B,
$$
so we obtain $ \frac{dY_j}{dx}=\frac{dY_j}{dT} \cdot
\frac{dT}{dx}=\frac{\frac{dY_j}{dT}}{\frac{dx}{dT}}=\frac{A Y_j +
B }{x - x_j} $ and we have
$$
\frac{Y'_j}{\sqrt{-1}r_j \exp(\sqrt{-1} \theta)}= \frac{Y'_j}{x'}=
\frac{A Y_j + B}{r_j \exp(\sqrt{-1} \theta)}
$$
therefore $Y'_j=\sqrt{-1}A Y_j + \sqrt{-1}B.$ Observe that the
solution of $$Y'_j=\sqrt{-1}AY_j + \sqrt{-1}B$$ with $Y_j(0)=
(u_j, v_j)$ is
$$
Y_j(\theta)= \exp(\sqrt{-1} \theta A) \cdot \left[\int^{t}_{0}
\exp(\sqrt{-1} s A) \cdot B(s) ds + Y_j(0)\right],
$$
that is,
$$
Y_j(\theta)=\left( -\frac{\mu^2}{4 \pi} \theta -\frac{1}{2}
\frac{\mu^2}{4 \pi^2} \theta^2 + u_j + \frac{\mu^2}{4 \pi^2}
\theta^2 + \frac{\mu}{2 \pi} v_j \theta, \frac{\mu}{2 \pi}\theta
+v_j \right)
$$
Therefore the holonomy is the biholomorphism $f:\Sigma_j \to
\Sigma_j$ defined by
$$
f(u_j, v_j)= Y_j(2 \pi) = (u_j + \mu v_j, v_j + \mu)
$$
and it proves the assertion.

Let us glue together the foliation on $A_j \times \bc P(2)$ and
the foliations on $D_j \times \bc P(2).$ First we consider $f_j$
of type P3 or R2. Then we  are in case (c) or (d). Observe that
$A_j \cap D_j$ is simply connected and $x_j \not\in A_j \cap D_j,$
we consider the coordinate system $(x, \widetilde{u}_j,
\widetilde{v}_j)$ in $(A_j \cap D_j) \times \bc P(2)$ such that
$$
\widetilde{u}_j = u_j \exp(-\alpha'_j
\log(\frac{x-x_j}{\frac{r}{2}}))
$$
and
$$
\widetilde{v}_j = v_j \exp(-\alpha''_j
\log(\frac{x-x_j}{\frac{r}{2}}))
$$
where $\log$ is the branch of the logarithm in $\ov\bc
\smallsetminus \{ x+ \sqrt{-1}y;\ x \leq 0 \}$ such that
$\log(1)=0.$ Observe that $x'_j = x_j + \frac{r}{2}$ implies that
$$
\widetilde{u}_j(x'_j, u_j) = u_j \exp(-\alpha'_j \log(1)) = u_j,\
\ \widetilde{v}_j(x'_j, v_j) = v_j
$$
and $\widetilde{u}_j(x,0)=0,\ \ \widetilde{v}_j(x,0)=0 .$
Therefore the leaves of the foliation on $(A_j \cap D_j) \times
\bc P(2)$ are given by $(\widetilde{u}_j,\widetilde{v}_j)=$
constant. Let us identify the point $(x, U_j, V_j) \in (A_j \cap
D_j) \times \C^2 \subset A_j \times \C^2$ with the point $((x,
u_j, v_j)) \in (A_j \cap D_j) \times \C^2 \subset D_j \times
\C^2,$ where
\begin{equation}\label{Equacao:Colagem1}
u_j = U_j \exp(\alpha'_j \log(\frac{x-x_j}{\frac{r}{2}}))
\end{equation}
and
\begin{equation}\label{Equacao:Colagem2}
v_j = V_j \exp(\alpha''_j \log(\frac{x-x_j}{\frac{r}{2}})).
\end{equation}

Notice that with equation~\ref{Equacao:Colagem1} and
equation~\ref{Equacao:Colagem2} we are gluing together in $(A_j
\cap D_j) \times \C^2$ plaques of the foliation $\tilde{\fa}$ on
$A_j \times \C^2$ with plaques of foliation $\widehat{\fa}$ on
$D_j \times \C^2.$ Observe that this identification sends the
fiber $\{x=c\} \subset A_j \times \C^2,$ $c \in A_j \cap D_j,$ in
the fiber $\{x=c\} \subset D_j \times \C^2,$ and the holonomy of
the curve $\beta_j=\alpha_j \ast \gamma_j \ast \alpha_j^{-1}$ in
the section $\Sigma''_j = \{x''_j\} \times \C^2 \subset A_j \times
\C^2$ with respect to the foliation obtained by gluing together
the $\widetilde{\fa}$ and $\widehat{\fa}$ is $(U_j, V_j) \mapsto
(\lambda'_j U_j, \lambda''_j V_j).$

If $f_j$ is of type P1R2 we are in case (e) and the
identifications \ref{Equacao:Colagem1} and  \ref{Equacao:Colagem2}
are
\begin{equation}\label{Equacao:Colagem11}
u_j = U_j - \frac{\nu_j}{2 \pi \sqrt{-1}} V_j \log\left(
\frac{x-x_j}{\frac{r}{2}} \right)
\end{equation}
\begin{equation}\label{Equacao:Colagem12}
v_j = V_j
\end{equation}

By analogy with it, if $f_j$ is of type P2 we are in case (b) and
the identifications are
\begin{equation}\label{Equacao:Colagem13}
u_j =\frac{1}{\lambda}(( U_j - \frac{\nu}{2 \pi \sqrt{-1}\mu} V_j)
\exp( \lambda \log ( \frac{x-x_j}{\frac{r}{2}} )))
\end{equation}
\begin{equation}\label{Equacao:Colagem14}
v_j = V_j \exp(\lambda \log( \frac{x-x_j}{\frac{r}{2}} ))
\end{equation}

And if $f_j$ is of type P1 we are in case (a) and the
identifications are
\begin{equation}\label{Equacao:Colagem15}
u_j = U_j + ( \frac{\mu}{2 \pi \sqrt{-1}} V_j + \frac{\mu^2}{(2
\pi \sqrt{-1})^2} \log( \frac{x-x_j}{\frac{r}{2}}) +
\frac{\mu^2}{4 \pi \sqrt{-1} } ) \cdot  \log (
\frac{x-x_j}{\frac{r}{2}} )
\end{equation}
\begin{equation}\label{Equacao:Colagem16}
v_j = V_j +  \frac{\mu}{2 \pi \sqrt{-1}} \log(
\frac{x-x_j}{\frac{r}{2}} )
\end{equation}

Now let us glue together the new foliation on $(A_j \cup D_j)
\times \C^2$ with the foliation on $(A_j \cap D_0) \times \C^2$
identify the points $(x, U_j, V_j) \in ((A_j \cup D_j) \cap D_0 )
\times \C^2$ with $(x, u_0, v_0) \in (A_j \cap D_0) \times \C^2
\subset D_0 \times \C^2$ by
\begin{equation}\label{Equacao:Colagem3}
u_0 = U_j \exp(\alpha'_0 \log(\frac{x}{x''_j}))
\end{equation}
and
\begin{equation}\label{Equacao:Colagem4}
v_0 = V_j \exp(\alpha''_0 \log(\frac{x}{x''_j})),
\end{equation}
where $\lambda'_0 = \exp(2 \pi \sqrt{-1} \alpha'_0)$ and
$\lambda''_0 = \exp(2 \pi \sqrt{-1} \alpha''_0),$ if $f_0$ is P3
or R2.

Notice that equation~\ref{Equacao:Colagem3} e
equation~\ref{Equacao:Colagem4} glue together plaques of the
foliation on $(A_j \cup D_j) \times \C^2$ with plaques of the
foliation on $(A_j \cap D_0) \times \C^2$ and this defines a new
foliation $\fa$ such that $D_0 \cup A_j \cup D_j$ is a leaf. The
holonomy of the curve $\beta_j$ in the section $\Sigma''_j =
\{x''_j\} \times \C^2 \subset D_0 \times \C^2$ is given by $(U_0,
V_0) \mapsto (\lambda'_j U_0, \lambda''_j V_0).$ Now suppose
$\gamma_0(\theta)= \frac{r}{2} \exp(\sqrt{-1} \theta),\ 0 \leq
\theta \leq 2 \pi,$ and for every $j=1, \dots, k$ let $\mu_j$ be
the segment of $\gamma_0$ between $x''_j$ and $\frac{r}{2}$ in the
positive sense. Let
$$
\delta_j= \mu_j \ast \beta_j \ast \mu_j^{-1}=\mu_j \ast \alpha_j
\ast \gamma_j \ast \alpha_j^{-1} \ast \mu_j^{-1},
$$
where $\gamma_j(\theta)=\frac{r}{2} \exp(\sqrt{-1} \theta) + x_j,\
\theta \in [0, 2 \pi]$ and $\Sigma_0 = \{ \frac{r}{2}\} \times
\C^2.$ The holonomy of the curve $\delta_j$ in $\Sigma_0$ is
$(U,V) \mapsto f_j(U, V).$

By analogy with this case, we have the identifications
$$
u_0 = U_j - \frac{\nu_j}{2 \pi \sqrt{-1}} V_j \log\left(
\frac{x}{x''_j} \right)
$$
$$
v_0 = V_j
$$
if $f_0$ is of type P1R2,
$$
u_0 =\frac{1}{\lambda}(( U_j - \frac{\nu}{2 \pi \sqrt{-1}\mu} V_j)
\exp( \lambda \log ( \frac{x-x_j}{\frac{r}{2}} )))
$$
$$
v_0 = V_j \exp(\lambda \log( \frac{x}{x''_j} ))
$$
if $f_0$ is of type P2 and
$$
u_0 = U_j + ( \frac{\mu}{2 \pi \sqrt{-1}} V_j + \frac{\mu^2}{(2
\pi \sqrt{-1})^2} \log( \frac{x}{x''_j}) + \frac{\mu^2}{4 \pi
\sqrt{-1} } ) \cdot  \log ( \frac{x}{x''_j} )
$$
$$
v_0 = V_j +  \frac{\mu}{2 \pi \sqrt{-1}} \log( \frac{x}{x''_j} )
$$
if $f_j$ is of type P1 respectively.

Now, let $\widetilde{M}=\ov\bc \times \C^2$ be a complex manifold
and let $\widetilde{\mathcal{F}}$ be a foliation obtained at the
end of the process. By construction the holonomy of the leaf $U =
\left( \bigcup_{i=0}^k A_i \right) \cap \left( \bigcup^k_{j=0} D_j
\right)$ in $\Sigma_0$ is generated by $f_1, \dots, f_k$ and the
holonomy of the curve $\delta_1 \ast \dots \ast \delta_k \ast
\gamma_0$  is the identity. Notice that $\widetilde{M}$ admits the
vertical foliation  $x=$ constant on $A_j \times \C^2,$ $D_j
\times \C^2,$ $D_0 \times \C^2$ and it cuts $U$ at a single point
and so we can define a projection $\widetilde{p}: \widetilde{M}
\to U$ such that $\widetilde{p}^{-1}(x)$ is the leaf of this new
foliation. Observe that this new foliation is transverse to
$\widetilde{\fa}$ in $\widetilde{M} \smallsetminus
\bigcup^k_{j=0}\{ x=x_j \}.$ Suppose the annulus $A=T \cap U,$ if
$\delta$ is a closed curve in $A$ which generates the homotopy of
$A,$ then the holonomy of $\delta$ with respect to
$\widetilde{\fa}$ in some transversal section is trivial, because
$\delta$ is homotopic to the curve $\delta_1 \ast \dots \ast
\delta_k \ast \gamma_0$ in $U \smallsetminus
\bigcup^k_{j=0}\{x_j\}$ and the holonomy of this is trivial. Then
we use the holonomy and obtain that the restricted foliation
$\widetilde{\mathcal{F}}|_{\widetilde{p}^{-1}(A)}$ is a product
foliation, that is, there is a biholomorphism $\varphi$ of the
some neighborhood $W$ of $A \subset \widetilde{p}^{-1}(A)$ onto $A
\times \Delta,$ where $\Delta \subset \C^2$ is a polydisc such
that it sends leaves of $\widetilde{\mathcal{F}}|_W$ onto leaves
$A \times \{c\},\ c \in \Delta$ of the trivial foliation.

Now we glue together the foliations $\widetilde{\fa}$ in
$\widetilde{M}$ an $\widehat{\mathcal{F}}$ in $V \times D$ by
$\varphi.$

Observe that we use the same ideas in the local model of foliation
in other affine coordinates of $\bc P(2)$ and it is proves the
result.

\end{proof}

\section{Normal forms of Riccati foliations}
Now we prove
Theorems~\ref{Theorem:ClassificacaoRiccatiPotenciasEsferaRiemann}
and \ref{Theorem:ClassificacaoRiccatiEspacoProjetivo}.

\begin{proof}[Proof of
Theorem~\ref{Theorem:ClassificacaoRiccatiPotenciasEsferaRiemann}]

We consider a singular holomorphic foliation $\fa$ on $\ov\bc
\times\ov\bc^n$ given by a polynomial vector field $X$ in affine
coordinates $(x,y)\in \C \times\C^n \hookrightarrow \ov\bc
\times\ov\bc^n$ and assume that  $\fa$ is transverse to almost
every fiber of $\eta$. Write $
X(x,y)=P(x,y)\frac{\partial}{\partial x} +
Q_1(x,y)\frac{\partial}{\partial y_1} + \dots +
Q_n(x,y)\frac{\partial}{\partial y_n},$. If $\{x_0' \} \times
\ov\bc^n$ is not an invariant fiber by $\fa$, then $\fa$ by
compactness $\fa$ is transverse to $\{ x \} \times \ov\bc^n$,
$\forall x\approx x_0',$. Hence $P(x,y) \neq 0$, $\forall x\approx
x_0',$ $\forall y \in \C^n$ and thus $P(x,y)=p(x)$.

\begin{Claim}
We have $\deg_{y_n}(Q_n) \leq 2$ {\rm(}where $\deg_{y_n}(\cdot)$
denotes the
 degree with respect to the variable $y_n${\rm)}.
 \end{Claim}
\begin{proof} Suppose $\deg_{y_n}(Q_n)
> 2$  and write  $\deg_{y_n}(Q_n) =
m + 2$ for some $m\in \mathbb N$. The foliation $\fa$ is given by
the meromorphic vector field
$$
X = p \frac{\partial}{\partial x} + Q_1 \frac{\partial}{\partial
y_1} + \dots + Q_{n-1} \frac{\partial}{\partial y_{n-1}} - w_n^2
\frac{1}{w_n^{m+2}} \tilde{Q}_n \frac{\partial}{\partial w_n}
$$
in affine coordinates $(x,y_1,\dots,y_{n-1},w_n)\in \C \times
\C^n$ where $w_n=\frac{1}{y_n}$ and $\tilde{Q}_n$ is a polynomial
in $\C[x,y_1,\dots,y_{n-1},w_n].$ If we multiply it by $w^m_n$ ,
we obtain a polynomial vector field as
$$
\tilde X = w_n^m p \frac{\partial}{\partial x} + w_n^m Q_1
\frac{\partial}{\partial y_1} + \dots + w_n^m Q_{n-1}
\frac{\partial}{\partial y_{n-1}} + \tilde{Q}_n
\frac{\partial}{\partial w_n}
$$
and we obtain three possibilities:

\noindent {\bf 1}. if $\deg_{w_n}(Q_1)=m$, then
$$
\tilde X = w_n^m p \frac{\partial}{\partial x} + \tilde{Q}_1
\frac{\partial}{\partial y_1} + \dots + w_n^m Q_{n-1}
\frac{\partial}{\partial y_{n-1}} + \tilde{Q}_n
\frac{\partial}{\partial w_n},
$$
{\bf 2}. if $\deg_{w_n}(Q_1) > m$, that is, $\exists l \in
\mathbb{N}$ such that $\deg_{w_n}(Q_1) = m+l$, then
$$
\tilde X = w_n^m p \frac{\partial}{\partial x} + w_n^m
\frac{1}{w_n^{m+l}}\tilde{Q}_1 \frac{\partial}{\partial y_1} +
\dots + w_n^m Q_{n-1} \frac{\partial}{\partial y_{n-1}} +
\tilde{Q}_n \frac{\partial}{\partial w_n}.
$$
We multiply it by $w_n^l$ and obtain a new polynomial vector field
$$
\tilde{\tilde X} = w_n^{m+l} p \frac{\partial}{\partial x} +
\tilde{Q}_1 \frac{\partial}{\partial y_1} + \dots + w_n^{m+l}
Q_{n-1} \frac{\partial}{\partial y_{n-1}} + w_n^l \tilde{Q}_n
\frac{\partial}{\partial w_n},
$$
{\bf 3}. If $\deg_{w_n}(Q_1) < m$, that is, $\exists l \in
\mathbb{N}$ such that $\deg_{w_n}(Q_1) = m-l$, then
$$
\tilde X = w_n^m p \frac{\partial}{\partial x} + w_n^l \tilde{Q}_1
\frac{\partial}{\partial y_1} + \dots + w_n^m Q_{n-1}
\frac{\partial}{\partial y_{n-1}} + \tilde{Q}_n
\frac{\partial}{\partial w_n}
$$
where $\tilde{Q}_1$, is a polynomial in
$\C[x,y_1,\dots,y_{n-1},w_n].$ In all these cases we use the same
ideas for all $Q_j$ with $j=2, \dots, n-1$ and it implies a
polynomial vector field $X$ without poles in affine coordinates
$(x,y_1,\dots,y_{n-1},w_n)\in \C \times \C^n \hookrightarrow
\ov\bc \times \ov\bc^n$ such that $\fa$ is tangent to the fiber
$\{x_0\} \times \ov\bc^n$ in the set $\{ x=x_0,\ w_n=0 \}$, and so
it is a contradiction. \end{proof}

Therefore
$$
X = p \frac{\partial}{\partial x} + Q_1 \frac{\partial}{\partial
y_1} + \dots + Q_{n-1} \frac{\partial}{\partial y_{n-1}} + q
\frac{\partial}{\partial y_n}
$$
where $q = q_{n,2} y_n^2 + q_{n,1} y_n + q_{n,0}$ with $q_{n,2},
q_{n,1}, q_{n,0} \in \C[x,y_1,\dots,y_{n-1}]$.

Now observe that $\deg_{y_n}(Q_j)=0$ for all $j=1,\dots, n-1$. In
fact, assume that $\deg_{y_n}(Q_j)>0$ so there exists $l_j \in
\mathbb{N}$ such that $\deg_{y_n}(Q_j)=l_j,$ then
$$
X = p \frac{\partial}{\partial x} + Q_1 \frac{\partial}{\partial
y_1} + \cdots + Q_{n-1} \frac{\partial}{\partial y_{n-1}} -
\tilde{q} \frac{\partial}{\partial w_n}
$$
where $\tilde{q} = q_{n,2} + q_{n,1} w_n + q_{n,0} w_n^2$ with
$q_{n,2}$, $q_{n,1}$ and $q_{n,0}$ are polynomials in
$\C[x,y_1,\dots,y_{n-1}]$. Let $m = \max\{\deg_{y_n}(Q_j);\
j=1,\dots,n-1\}.$ Suppose $m = \deg_{y_n}(Q_1)$. We multiply the
vector field $X$ by $w^m_n$ so we obtain the vector field
$$
X = w_n^m p \frac{\partial}{\partial x} + \widetilde{Q}_1
\frac{\partial}{\partial y_1} + w_n^{\alpha_2} \widetilde{Q}_2
\frac{\partial}{\partial y_2} + \cdots + w_n^{\alpha_{n-1}}
\widetilde{Q}_{n-1} \frac{\partial}{\partial y_{n-1}} - w_n^m
\tilde{q} \frac{\partial}{\partial w_n}
$$
with $\alpha_2 = m - \deg_{y_n}(Q_2)$ and $\alpha_{n-1} = m -
\deg_{y_n}(Q_{n-1})$, where $\widetilde{Q}_j$ are polynomials in
$\C[x,y_1,\dots,y_{n-1},w_n]$ and $q_{n,k}$ are polynomials in
$\C[x,y_1,\dots,y_{n-1}]$. We obtain $\fa$ is tangent to $\{x_0\}
\times \ov\bc^n$ in the set $\{ x=x_0,\ w_n=0 \}$ so we have a
contradiction. We obtain
$$
X = p \frac{\partial}{\partial x} + Q_1 \frac{\partial}{\partial
y_1} + \dots + Q_{n-1} \frac{\partial}{\partial y_{n-1}} + q
\frac{\partial}{\partial y_n}
$$
where $q = q_{n,2} y_n^2 + q_{n,1} y_n + q_{n,0}$ with $q_{n,2}$,
$q_{n,1}$, $q_{n,0}$, $Q_1$, $\dots$, $Q_{n-1}$ polynomials in
$\C[x,y_1,\dots,y_{n-1}]$ and $p$ is polynomial in $\C[x]$. By
analogy with it we use the affine coordinates $(x,y_1,\dots,
w_k,\dots,y_{n-1},y_n) \in \C^{n+1} \hookrightarrow \ov\bc \times
\ov\bc^n,$ with $w_{k}=\frac{1}{y_{k}},$ $k \in \{1,\dots,n-1\}$
and we obtain $\deg_{y_{k}}(Q_{k})\leq 2$, $\deg_{y_{k}}(Q_j) = 0$
for every $j \in \{1,\dots,n\}\smallsetminus\{k\}.$ Therefore
$$
X = p \frac{\partial}{\partial x} + Q_1 \frac{\partial}{\partial
y_1} + \cdots + Q_n \frac{\partial}{\partial y_n}
$$
in affine coordinates $(x,y)\in \C\times\C^n \hookrightarrow
\ov\bc \times \ov\bc^n,$ $y=(y_1,\dots,y_n),$ where $Q_j =
q_{j,2}(x)y_j^2 + q_{j,1}(x)y_j + q_{j,0}(x)$ with $q_{j,2}$,
$q_{j,1}$, $q_{j,0}$, $p$ polynomials in $\C[x]$. And we prove
Theorem~\ref{Theorem:ClassificacaoRiccatiPotenciasEsferaRiemann}.
\end{proof}

\begin{Corollary}\label{Corollary:Scardua}
Let $\fa$ be a singular holomorphic foliation on $\ov\bc \times
\ov\bc^n$ given by a polynomial vector field $X$ in affine
coordinates on $\C \times \C^{n}.$ Suppose $\fa$ is transverse at
least one fiber $\{x_0\} \times\ov\bc^n$ of $\eta.$ Then $\fa$ is
a Riccati foliation on $\ov\bc \times \ov\bc^n.$
\end{Corollary}

\begin{proof}
We consider
$$
X(x,y)=P(x,y)\frac{\partial}{\partial
x}+Q_1(x,y)\frac{\partial}{\partial
y_1}+\cdots+Q_n(x,y)\frac{\partial}{\partial y_n}
$$
in affine coordinates $(x,y) \in \C \times \C^n \hookrightarrow
\ov\bc \times \ov\bc^n,$ $y=(y_1,\dots,y_n)$ where $P$, $Q_j$ are
polynomials in $\C[x,y].$ If $\fa$ is transverse to $\{x_0\}\times
\ov\bc^n,$ then $P(x,y)=p(x),$ $\forall x \in \C.$ We use the
transversality and it implies that there exists $\varepsilon > 0$
such that $\fa$ is transverse to $\{x_0\} \times \ov\bc^n,$
$\forall x \in D_{\varepsilon}(x_0).$ Now we use the
Theorem~\ref{Theorem:ClassificacaoRiccatiPotenciasEsferaRiemann}.
It implies that $\deg_{y_n}(Q_n) \leq 2$ and $\deg_{y_n}(Q_j)=0$,
$j \in \{ 1, \dots, n-1 \}.$ Then $\frac{\partial^k Q_j}{\partial
y_n^k}=0$ and $\frac{\partial Q_n}{\partial y_n}=0,$ $\forall k
\geq 3,$ $\forall x \in  D_{\varepsilon}(x_0).$ Now we use the
identity theorem \cite{gunning1} and obtain $\frac{\partial^k
Q_j}{\partial y_n^k}=0$ and $\frac{\partial Q_n}{\partial y_n}=0$
in $\ov\bc \times \ov\bc^n$ for every $k \geq 3.$ By analogy with
this we conclude the proof.
\end{proof}

\begin{proof}[Proof of
Theorem~\ref{Theorem:ClassificacaoRiccatiEspacoProjetivo}] Suppose
$\fa$ is a singular holomorphic foliation on $\ov\bc \times \bc
P(2)$ given by a polynomial vector field
$X(x,y,z)=P(x,y,z)\frac{\partial}{\partial
x}+Q(x,y,z)\frac{\partial}{\partial
y}+R(x,y,z)\frac{\partial}{\partial z}$ in affine coordinates
$(x,y,z)\in \C \times \C^2 \subset \ov\bc \times \bc P(2).$ and
transverse to almost every fiber of $\eta$. Let $\{x_0\}\times \bc
P(2)$ be a fiber transverse to  $\fa$. Then $\fa$ is transverse to
$\{x\} \times \bc P(2),$ $\forall x$ in a neighborhood of $x_0.$
This implies $P(x,y,z)=p(x),$ $\forall (x,y,z) \in \C \times
\C^2.$ On the other hand,
$$
X(x,u,v)=p(x)\frac{\partial}{\partial x}-u^2
Q(x,\frac{1}{u},\frac{v}{u})\frac{\partial}{\partial u}+ ( u
R(x,\frac{1}{u},\frac{v}{u}) -uv Q(x,\frac{1}{u},\frac{v}{u})
)\frac{\partial}{\partial v}.
$$
in affine coordinates $(x,u,v)\in\C \times\C^2 \hookrightarrow
\ov\bc \times\bc P(2)$ with $u=\frac{1}{y}$ and $v=\frac{z}{y}.$
Therefore
$$
X(x,u,v) = p(x)\frac{\partial}{\partial x}-u^2 \frac{1}{u^\alpha}
\widetilde{Q}(x,u,v)\frac{\partial}{\partial u} + ( u
\frac{1}{u^\beta} \widetilde{R}(x,u,v) -uv \frac{1}{u^{\alpha}}
\widetilde{Q}(x,u,v) )\frac{\partial}{\partial v}
$$
where $\widetilde{Q}$, $\widetilde{R}$ are polynomials in
$\C[x,u,v], \, \,  \alpha=\max\{m+n;
Q(x,y,z)=\displaystyle\sum_{l,m,n}q_{l,m,n}x^l y^m z^n\} $ and
$\beta=\max\{m+n; R(x,y,z)=\displaystyle\sum_{l,m,n}r_{l,m,n}x^l
y^m z^n\}.$ If $\{x_1\} \times \bc P(2)$ is a fiber transverse to
$\fa,$ then $\beta \leq \alpha.$ In fact, if $\beta > \alpha,$
then
$$
X(x,u,v) = p(x)\frac{\partial}{\partial x}-u^2 \frac{1}{u^\alpha}
\widetilde{Q}(x,u,v)\frac{\partial}{\partial u} +
\frac{1}{u^{\beta-1}}(\widetilde{R}(x,u,v) -u^{\beta - \alpha}v
\widetilde{Q}(x,u,v) )\frac{\partial}{\partial v},
$$
and we recall that $\beta > \alpha$ implies $\beta > 1.$ We
multiply the vector field $X$ by $u^{\beta-1}$  to obtain
$$
X(x,u,v) = u^{\beta-1}p(x)\frac{\partial}{\partial x}-u
u^{\beta-\alpha} \widetilde{Q}(x,u,v)\frac{\partial}{\partial u} +
(\widetilde{R}(x,u,v) -u^{\beta - \alpha}v \widetilde{Q}(x,u,v)
)\frac{\partial}{\partial v}.
$$
We observe that $X$ is not transverse to $\{ x_1 \} \times \C^2$
at the point $(x_1, 0,v_0),$ because if
$q(v)=\widetilde{R}(x_1,0,v)$ is the zero polynomial in $\bc [v],$
then $(x_1, 0,v_0)$ is a singularity of $X.$ Otherwise $q(v)$ is a
polynomial in $\bc [v] \smallsetminus \{ 0 \}$ and the point
$(x_1, 0,v_0)$ is a singularity of $X$ or $X$ is tangent to $\{
x_1 \} \times \bc P(2)$ at this point if $v_0$ is a zero of $q(v)$
or it is not. Therefore $\beta \leq \alpha,$ and
$$
X = p(x)\frac{\partial}{\partial x}-u^2 \frac{1}{u^\alpha}
\widetilde{Q}(x,u,v)\frac{\partial}{\partial u} +
\frac{1}{u^{\alpha}}( u u^{\alpha-\beta} \frac{1}{u^\beta}
\widetilde{R}(x,u,v) -uv \widetilde{Q}(x,u,v)
)\frac{\partial}{\partial v}.
$$
in affine coordinates $(x,u,v)\in\C \times\C^2 \hookrightarrow
\ov\bc \times\bc P(2).$ Notice that $\alpha \leq 2.$ In fact, if
$\alpha > 2,$ that is, there is $k \in \mathbb{Z} \smallsetminus
\{ -1,-2,\dots \}$ such that $\alpha= 3+k.$ Then
$$
X(x,u,v) = p(x)\frac{\partial}{\partial x}- \frac{1}{u^{k+1}}
\widetilde{Q}(x,u,v)\frac{\partial}{\partial u} +
\frac{1}{u^{k+2}}( u^{\alpha-\beta} \widetilde{R}(x,u,v) -v
\widetilde{Q}(x,u,v) )\frac{\partial}{\partial v}
$$
and we multiply it by $u^{k+1}$
$$
X(x,u,v)= u^{k+1} p(x)\frac{\partial}{\partial x}-
\widetilde{Q}(x,u,v)\frac{\partial}{\partial u} +
\frac{1}{u}(u^{\alpha-\beta} \widetilde{R}(x,u,v) -v
\widetilde{Q}(x,u,v) )\frac{\partial}{\partial v}.
$$
We obtain two possibilities: the polynomial $S$ in $\C[x,u,v]$
defined by
$$
S(x,u,v)=u^{\alpha-\beta} \widetilde{R}(x,u,v) -v
\widetilde{Q}(x,u,v)
$$
is multiple of $u$ or it is not.

{\bf Case (i).} If $S$ is multiple of $u,$ then $\fa$ is given by
the holomorphic vector field
$$
X(x,u,v)= u^{k+1} p(x)\frac{\partial}{\partial x}-
\widetilde{Q}(x,u,v)\frac{\partial}{\partial u} +
(u^{\alpha-\beta} \widetilde{R}(x,u,v) -v \widetilde{Q}(x,u,v)
)\frac{\partial}{\partial v}.
$$
and observe that $X$ is not transverse to $\{x_1\} \times \bc
P(2)$ in $(x_1, 0,v_0).$

{\bf Case (ii).} If $S$ is not multiple of $u,$ then $\fa$ is
given by the vector field without poles
$$
X(x,u,v)= u^{k+2} p(x)\frac{\partial}{\partial x}- u
\widetilde{Q}(x,u,v)\frac{\partial}{\partial u} + S(x,u,v)
\frac{\partial}{\partial v}
$$
and observe that $X$ is not transverse to $\{x_1\} \times \bc
P(2)$ at $(x_1, 0,v_0),$ either. Then it is a contradiction in
both  cases. Recall that $\beta \leq \alpha \leq 2.$ We obtain the
following possibilities:

{\bf Possibility 1.} If $\alpha=\beta=0,$ then
\begin{equation} \label{Equacao:fim1}
X(x,y,z)=p(x)\frac{\partial}{\partial  x}+
q(x)\frac{\partial}{\partial y}+ r(x)\frac{\partial}{\partial z}.
\end{equation}
in affine coordinates $(x,y,z)\in \C^3 \hookrightarrow \ov\bc
\times\bc P(2).$

{\bf Possibility 2.} If $\alpha=1,\ \beta=0,$ then
\begin{equation} \label{Equacao:fim2}
X(x,y,z)=p(x)\frac{\partial}{\partial  x}+
(A(x)+B(x)y+C(x)z)\frac{\partial}{\partial y}+
r(x)\frac{\partial}{\partial z},
\end{equation}
where $A,B,C \in \C[x].$

{\bf Possibility 3.} Se $\alpha=\beta=1,$ then
\begin{equation} \label{Equacao:fim3}
X(x,y,z)=p(x)\frac{\partial}{\partial  x}+
(A(x)+B(x)y+C(x)z)\frac{\partial}{\partial y}+
(a(x)+b(x)y+c(x)z)\frac{\partial}{\partial z},
\end{equation}
where $a,A,b,B,c,C \in \C[x].$

Observe that the foliation $\fa$ given by the vector field $X$
defined by equation~\ref{Equacao:fim3} (\ref{Equacao:fim1} or
~\ref{Equacao:fim2} respectively) is transverse to almost every
fiber of $\eta.$ In fact, we obtain
$$
X(x,u,v)= p(x)\frac{\partial}{\partial x}-u
(A(x)u+C(x)v+B(x))\frac{\partial}{\partial u} +
\tilde{R}(x,u,v)\frac{\partial}{\partial v},
$$
in affine coordinates $(x,u,v)\in \C^3\hookrightarrow \ov\bc
\times\bc P(2)$ where
$\tilde{R}(x,u,v)=a(x)u+c(x)v+b(x)-v(A(x)u+C(x)v+B(x))$ and so
$\fa$ is transverse to $\{x\}\times \bc P(2)$ if $p(x) \neq 0.$ On
the other hand, we obtain
$$
X(x,t,s)= p(x)\frac{\partial}{\partial x} +
\tilde{Q}(x,t,s)\frac{\partial}{\partial t}
-s(a(x)s+b(x)t+c(x))\frac{\partial}{\partial s}
$$
with $\tilde{Q}(x,t,s)=A(x)s+B(x)t+C(x)-t(a(x)s+b(x)t+c(x))$ in
affine coordinates $(x,t,s)\in \C^3 \hookrightarrow \C \times \bc
P(2)$ where $s=\frac{1}{z}$ and $t=\frac{y}{z}.$ In this case
$\fa$ is transverse to $\{x\}\times \bc P(2)$ if $p(x) \neq 0.$

{\bf Possibility 4.} If $\alpha=2,\ \beta=0,$ then
\begin{equation} \label{Equacao:fim4}
X(x,y,z)=p(x)\frac{\partial}{\partial x}+
Q(x,y,z)\frac{\partial}{\partial y}+ r(x)\frac{\partial}{\partial
z},
\end{equation}
where $Q(x,y,z)=A(x)+B(x)y+C(x)z+D(x)yz+E(x)y^2+F(x)z^2$ and $A,
B,$ $ C, D,$ $E, F\in \C[x].$

We obtain
$$
X(x,u,v)= p(x)\frac{\partial}{\partial x}-
\widetilde{Q}(x,u,v)\frac{\partial}{\partial u} + (u\  r(x) -
\frac{v}{u}\  \widetilde{Q}(x,u,v) )\frac{\partial}{\partial v},
$$
where $\tilde{Q}(x,u,v)=A(x)u^2+F(x)v^2+C(x)uv+D(x)v+B(x)u+E(x).$
Recall that $D, E$ or $F$ is not the zero polynomial. Then
$\tilde{Q}$ is not multiple of $u.$ Moreover the polynomial
$S(x,u,v)=u^2\ r(x) - v\ \widetilde{Q}(x,u,v)$ is not the zero
polynomial. We multiply the vector field $X$ by $u$ and obtain
$$
X(x,u,v)= u\ p(x)\frac{\partial}{\partial x} - u\
\widetilde{Q}(x,u,v)\frac{\partial}{\partial u} +
S(x,u,v)\frac{\partial}{\partial v}.
$$
Notice that $X$ is not transverse tov $\{x_1\} \times \C^2$ in
$(x_1, 0,v_0).$ This is a contradiction. Thus possibility 4 does
not occur.

{\bf Possibility 5.} If $\alpha=2,\ \beta=1,$ then
\begin{equation} \label{Equacao:fim5}
X(x,y,z)=p(x)\frac{\partial}{\partial x}+
Q(x,y,z)\frac{\partial}{\partial y}+
R(x,y,z)\frac{\partial}{\partial z},
\end{equation}
where $Q(x,y,z)=A(x)+B(x)y+C(x)z+D(x)yz+E(x)y^2+F(x)z^2,$
$R(x,y,z)=a(x)+b(x)y+c(x)z$ e $a, b, c, A, B,$ $ C, D,$ $E, F \in
\C[x].$ Then
$$
X(x,u,v)= p(x)\frac{\partial}{\partial x}-
\widetilde{Q}(x,u,v)\frac{\partial}{\partial u} + \frac{1}{u}(u\
\tilde{R}(x,u,v) - v\  \widetilde{Q}(x,u,v)
)\frac{\partial}{\partial v},
$$
where $\tilde{Q}(x,u,v)=A(x)u^2+F(x)v^2+C(x)uv+D(x)v+B(x)u+E(x)$ e
$\tilde{R}=a(x)u+c(x)v+b(x),$ in affine coordinates $(x,u,v)\in\C
\times\C^2 \hookrightarrow \ov\bc \times \bc P(2).$ Observe that
$D, E$ or $ F$ is not the zero polynomial. Therefore $\tilde{Q}$
is not multiple of $u,$ and the polynomial $S(x,u,v)=u\
\tilde{R}(x,y,z) - v\ \widetilde{Q}(x,u,v)$ is not the zero
polynomial. We multiply $X$ by $u$ and obtain
$$
X(x,u,v)= u\ p(x)\frac{\partial}{\partial x} - u\
\widetilde{Q}(x,u,v)\frac{\partial}{\partial u} +
S(x,u,v)\frac{\partial}{\partial v},
$$
and observe that $X$ is not transverse to $\{x_1\} \times \C^2$ in
$(x_1, 0,v_0).$ Then it is a contradiction and possibility 5 does
not occur.

{\bf Possibility 6.} If $\alpha=\beta=2,$ then
\begin{equation} \label{Equacao:fim6}
X(x,y,z)=p(x)\frac{\partial}{\partial x}+
Q(x,y,z)\frac{\partial}{\partial y}+
R(x,y,z)\frac{\partial}{\partial z},
\end{equation}
where $Q(x,y,z)=A(x)+B(x)y+C(x)z+D(x)yz+E(x)y^2+F(x)z^2,$
$R(x,y,z)=a(x)+b(x)y+c(x)z+d(x)yz+e(x)y^2+f(x)z^2$ and $a, b, c,$
$d, e, f,$ $A, B,$ $ C, D,$ $E, F\in \C[x].$ Then
$$
X(x,u,v)= p(x)\frac{\partial}{\partial x}-
\widetilde{Q}(x,u,v)\frac{\partial}{\partial u} +
\frac{1}{u}(\tilde{R}(x,u,v) - v\  \widetilde{Q}(x,u,v)
)\frac{\partial}{\partial v},
$$
in affine coordinates $(x,u,v)\in\C \times\C^2 \hookrightarrow
\ov\bc \times \bc P(2)$ where
$\tilde{Q}(x,u,v)=A(x)u^2+F(x)v^2+C(x)uv+D(x)v+B(x)u+E(x)$ and
$\tilde{R}=a(x)u^2+f(x)v^2+c(x)uv+d(x)v+b(x)u+e(x).$ Notice that
none of the polynomials  $D, E$ or $F\in \C[x]$ is identically
zero  and the same holds for  the polynomials $d, e$ or $f \in
\C[x].$ Therefore the polynomial $S(x,u,v)=\tilde{R}(x,u,v) - v\
\widetilde{Q}(x,u,v)$ is a multiple of $u$ if and only if $e=0$,
$d=E,$ $f=D$ and $F=0.$ Then the foliation $\fa$ is given by $X$
is transverse to $\{x\} \times \bc P(2)$ if $p(x)\neq 0.$

On the other hand
$$
X(x,t,s)= p(x)\frac{\partial}{\partial x}+
\frac{1}{s}(\tilde{Q}(x,t,s)-t\tilde{R}(x,t,s))
\frac{\partial}{\partial t} -
\widetilde{Q}(x,t,s)\frac{\partial}{\partial s},
$$
in affine coordinates $(x,t,s)\in \C \times \C^2 \hookrightarrow
\ov\bc \times \bc P(2)$ with $t=\frac{y}{z}$ e $s=\frac{1}{z}$
where
$$
\widetilde{Q}(x,t,s)=A(x)s^2+E(x)t^2+B(x)ts+C(x)s+D(x)t
$$
and
$$
\tilde{R}(x,t,s)=a(x)s^2+b(x)ts+c(x)s+E(x)t+D(x).
$$
Therefore
$$
\tilde{Q}(x,t,s)-t\tilde{R}(x,t,s)=-ts^2a(x)-t^2sb(x)+ts(B-c)(x)+s^2A(x)+sC(x),
$$
and the foliation $\fa$ is given by the vector field without poles
$$
X(x,t,s)= p(x)\frac{\partial}{\partial x}+ U(x,t,s)
\frac{\partial}{\partial t} -
\widetilde{Q}(x,t,s)\frac{\partial}{\partial s},
$$
with $U(x,t,s)=-ts\ a(x)-t^2\ b(x)+t(B-c)(x)+s\ A(x)+C(x),$ and it
is transverse to $\{x\}\times \bc P(2)$ if $p(x)\neq 0.$

If $S$ is not multiple of $u,$ then the foliation $\fa$ is given
by the vector field without poles
$$
X(x,u,v)= u\ p(x)\frac{\partial}{\partial x} -u\
\widetilde{Q}(x,u,v)\frac{\partial}{\partial u} +
S(x,u,v)\frac{\partial}{\partial v},
$$
and it is not transverse to $\{x_1\} \times \C^2$ in $(x_1, 0,
v_0).$  We obtain a contradiction. This ends the proof of
Theorem~\ref{Theorem:ClassificacaoRiccatiEspacoProjetivo}.
\end{proof}

The above proof indeed gives:
\begin{Corollary}
Let $\fa$ be a singular holomorphic foliation on  $\ov\bc \times
\bc P(2)$ given by a polynomial vector field $X$ in affine
coordinates on $\C \times \C^2.$ Suppose $\fa $ is transverse at
least one fiber $\{x_0\} \times \bc P(2)$ of $\eta.$ Then $\fa$ is
a Riccati foliation on $\ov\bc \times \bc P(2).$
\end{Corollary}

\leftline{Instituto de Matem\'atica} \leftline{Universidade
Federal do Rio de Janeiro} \leftline{Caixa Postal 68530}
\leftline{CEP. 21945-970 Rio de Janeiro - RJ} \leftline{BRASIL}


\begin{thebibliography}{99}


\bibitem{Beardon} A. F. Beardon: The geometry of discrete groups. Corrected
reprint of the 1983 original. Graduate Texts in Mathematics, 91.
Springer-Verlag, New York, 1995.


\bibitem{camacho} Camacho, César; Lins Neto, Alcides. {\it Geometry theory of
foliations.} Translated from the Portuguese by Sue E. Goodman.
Birkhäuser Boston, Inc., Boston, MA, 1985. vi + 205 pp.

\bibitem{Camacho-Scardua} { Camacho, César;  Sc\'ardua, Bruno}: {\em
Holomorphic foliations with Liouvillian first integrals}, {Ergodic
Theory and Dynamical Systems} (2001), 21, pp.717-756.


\bibitem{Godbillon}  Godbillon, Claude: Feuilletages.  \'Etudes g\'eom\'etriques.
With a preface by G. Reeb. Progress in Mathematics, 98.
Birkh\"auser Verlag, Basel, 1991.



\bibitem{griffiths} Griffiths, Phillip; Harris, Joseph. {\it Principles of
algebraic geometry.} Reprint of the 1978 original. Wiley Classics
Library. John Wiley \& Sons, Inc., New York, 1994. xiv + 813 pp.


\bibitem{gunning1} Gunning, Robert C. {\it Introduction to holomorphic
functions of several variables. Vol. I. Function theory.} The
Wadsworth \& Brooks/Cole Mathematics Series. Wadsworth \&
Brooks/Cole Advanced Books \& Software, Pacific Grove, CA, 1990.
xx + 203 pp.

\bibitem{cidinho} Lins Neto, Alcides. {\it Construction of singular holomorphic
vector fields and foliations in dimension two.} J. Differential
Geom. 26  (1987), no. 1, 1 - 31.

\bibitem{linsnetoscardua} Lins Neto, A.; Scárdua, B. C. A. {\it Folheações
algébricas complexas.} 21° Colóquio Brasileiro de Matemática.
Instituto Nacional de Matemática Pura e Aplicada \& CNPq, Rio de
Janeiro, RJ, 1997. ii + 192 pp.

\bibitem{okamoto} Okamoto, Kazuo. {\it Sur les feuilletages associés aux
équations du second ordre à points critiques fixes de P.
Painlevé.} (French) [On foliations associated with fixed critical
points] Japan. J. Math. (N. S.) 5 (1979), no. 1, 1 - 79.

\bibitem{Pan-Sebastiani2}  Pan,I. and   Sebastiani, M.
Les Équations Différentielles Algébriques et les Singularités
Mobiles.  Monografias de Matem\'atica do  IMPA. Rio de Janeiro,
Brazil, 2005.


\bibitem{Pan-Sebastiani}  Pan,I. and   Sebastiani, M.
Sur les \'equations diff\'erentielles alg\'ebriques admettant des
solutions avec une singularit\'e essentielle.  Ann. Inst. Fourier
(Grenoble) 51 (2001), no. 6, 1621--1633.


\bibitem{Scardua0}  Sc\'ardua, Bruno:   Transversely affine
  and transversely projective holomorphic foliations.
  Ann. Sci. École Norm. Sup. (4) 30 (1997), no. 2, 169--204.



\bibitem{scardua1} Scárdua, Bruno:  On complex codimention-one
foliations transverse fibrations. J. Dyn. Control Syst. 11 (2005),
no. 4, 575-603.

\bibitem{Scardua5}  Sc\'ardua, Bruno. Holomorphic
 foliations transverse to fibrations on hyperbolic manifolds.
 Complex Variables Theory Appl. 46 (2001), no. 3, 219--240.

\end{thebibliography}
\end{document}